\documentclass[12pt]{article}
\usepackage{amsmath,latexsym,amssymb,amsthm,enumerate,amsmath,amscd}
\usepackage{amsthm,amsmath,amsfonts,amssymb,amscd,latexsym,stmaryrd}
\usepackage[vcentermath,enableskew]{youngtab}
\usepackage[mathscr]{eucal}
\usepackage{color}
\definecolor{med-gray}{gray}{0.5}
\definecolor{gray1}{gray}{0.85}
\definecolor{gray2}{gray}{0.90}
\definecolor{gray3}{gray}{0.64}
\definecolor{gray4}{gray}{0.55}
\definecolor{verylight-yellow}{rgb}{1,1,0.7}
\definecolor{yellow}{rgb}{1,1,0.2}
\definecolor{vivid-blue}{rgb}{0.2,0,1}
\definecolor{light-pink}{rgb}{1,0.8,1}
\definecolor{med-pink}{rgb}{1,0.6,1}
\definecolor{aqua}{rgb}{0.0, 1.0, 1.0}
\definecolor{light-gray}{rgb}{0.5, 0.9, 0.5}
\usepackage[usenames,dvipsnames]{xcolor}
\usepackage{colortbl}
\usepackage[all,cmtip]{xy}
\usepackage{comment}
\usepackage{array}
\usepackage{makeidx}
\usepackage{amssymb,multicol}
\usepackage{fancyhdr}
\usepackage{fancybox}
\usepackage{eurosym}
\theoremstyle{definition}
\usepackage{xcolor}
\usepackage{colortbl}
\usepackage[draft]{pdfpages}
\usepackage{textcomp}
\usepackage{graphicx}
\usepackage{tikz}
\usepackage{ytableau}
\usepackage{enumerate}

\theoremstyle{plain}
\newtheorem{theorem}{Theorem}[section]

\newtheorem{corollary}[theorem]{Corollary}
\newtheorem{lemma}[theorem]{Lemma}
\theoremstyle{definition}
\newtheorem{definition}[theorem]{Definition}

\newtheorem{remark}[theorem]{Remark}

\newtheorem{example}[theorem]{Example}

\newtheorem{question}[theorem]{Question}
\newtheorem*{ack}{Acknowledgment}

\newtheorem{claim}[theorem]{Claim}

\newtheorem{thm}{Theorem}

\newtheorem{note}[theorem]{Note}

\numberwithin{equation}{section}
\numberwithin{table}{section} 
\DeclareMathOperator{\Gor}{Gor}
 
\DeclareMathOperator{\Grass}{\rm{Grass}}

\newcommand{\AAA}{\mathcal{A}}

\newcommand{\D}{\mathbf{D}}

\newcommand{\Hom}{\ensuremath{\mathit{Hom}}}

\newcommand{\kk}{\ensuremath{\mathsf{k}}}

\newcommand{\maxA}{\ensuremath{\mathfrak{m}_{A}}}
\newcommand{\maxR}{\ensuremath{\mathsf{m}_R}}

\DeclareMathOperator{\Ann}{\mathrm{Ann}}

\def\F{\mathcal{F}}
\def\Gr{\mathrm{Gr}}

\def\cha{\mathrm{char}\ }

\def\Hom{\mathrm{Hom}}

\def\Pgl{\mathrm{PGL}}

\def\Hilb{\mathrm{Hilb}}

\def\Grass{\mathrm{Grass}}

\def\<{\left<}
\def\>{\right>}

\def\B{\mathcal{B}}

\def\C{\mathcal{C}}

\def\ns{\footnotesize \it}

\def\max{\mathrm{max}}

\def\GGor{\mathrm{GGor}}
\def\Gr{\mathrm{Gr}}
\def\Gor{\mathrm{Gor}}
\def\A{A}
\def\D{\mathcal{D}}
\def\m{\mathfrak{m}}

\newlength\distvertices

\title{Reducibility of a family of local Artinian Gorenstein algebras}
\author{
Anthony Iarrobino\\[.05in]
{\ns Department of Mathematics, Northeastern University, Boston, MA 02115,
USA.
}\\[.2in] Pedro Macias Marques\\[.05in]
{\ns Departamento de Matem\'{a}tica, Escola de Ci\^{e}ncias e Tecnologia, Centro de Investiga\c{c}\~{a}o}\\[-.05in]
{\ns  em Matem\'{a}tica e Aplica\c{c}\~{o}es, Instituto de Investiga\c{c}\~{a}o e Forma\c{c}\~{a}o Avan\c{c}ada,}\\[-.05in]
{\ns Universidade de \'{E}vora, Rua Rom\~{a}o Ramalho, 59, P--7000--671 \'{E}vora, Portugal}\\
}

\date{December 27, 2021}
\begin{document}

\maketitle
\begin{abstract}
The Jordan type of an Artinian algebra is the Jordan block partition associated to multiplication by a generic element of the maximal ideal.  We study the Jordan type for Artinian Gorenstein (AG) local (non-graded) algebras $\A$, and the interaction of Jordan type with the symmetric decomposition of the Hilbert function $H(\A)$. We give  examples of Gorenstein sequences $H$ for which the family $\Gor(H)$ of AG algebras having Hilbert function $H$ has three irreducible components, each corresponding to a symmetric decomposition of $H$.  The component structure results from the intersection of two opposing filtrations of the family $\Gor(H)$ of AG algebras: that by Jordan type satisfies the usual dominance property; the second filtration, by symmetric decomposition, satisfies a known semicontinuity property. Our examples are in codimension three -- the lowest codimension of such an example, as $\Gor(H)$ is irreducible in codimension two.
We also compare the Jordan type of the AG algebra $\A$ and of its associated graded algebra $\A^\ast$, showing that they can be different. We explore a problem of realizability of partial symmetric decompositions.
\end{abstract}
\tableofcontents
\section{Introduction.}
Let $R={\mathsf{k}}\{x_1,\ldots, x_r\}$ be the regular local ring in $r$ variables over an infinite field $\mathsf{k}$; for simplicity we will assume in this paper that $R$ has standard grading.
We assume that $A$ is an Artinian Gorenstein (AG) local algebra $\A=R/I$ with maximal ideal $\maxA$ (or $\m$, when $A$ is understood) of socle degree $j=j(\A)$, such that $\A/\mathfrak{m}=\mathsf{k}$. Recall that the Hilbert function $H(\A)$ satisfies $H(\A)=(h_0,h_1,\ldots,h_j)$ where $h_i=H(\A)_i=\dim_{\mathsf{k}}(\m^i/\m^{i+1})$ and $h_j>0$.  We will say that $H=(1,r,\ldots , 1_j)$ is a \emph{Gorenstein sequence} of codimension $r$ if there is an AG local algebra $\A$, quotient of $R$ such that $H(\A)=H$. The algebras we study are in general non-homogeneous and our Gorenstein sequences $H$ will in general not be symmetric.\footnote{In the literature, a Gorenstein sequence $H$ is often assumed to be symmetric, in which case there is a homogeneous AG algebra with Hilbert function $H$, as $H(A)$ symmetric implies $A^\ast$ is AG (\cite[Proposition~9]{Wat} or \cite[Proposition 1.7]{I2}).} For $H=(1,r,\ldots, 1)$ a Gorenstein sequence, we will denote by $\Gor(H)$ the reduced variety whose closed points are the Gorenstein quotients $\A$ of $R$ having Hilbert function $H$.  In this paper we exhibit a new method for showing that certain
varieties $\Gor(H)$ for $H$ of codimension three, have several irreducible components -- each corresponding to a \emph{symmetric decomposition} of $H$.\par
\subsection{Basic Notions.}
The first author showed that the associated graded algebra $\A^\ast=\Gr_{\mathfrak{m}}(\A)$ has a filtration by ideals $C_{\A}(a)$ whose successive quotients $Q_A(a)$ are reflexive ${\kk}$ modules. We first define these ideals and subquotients of $A^\ast$, that we term the \emph{symmetric decomposition} for $A$. This is our first basic ingredient.
\begin{definition}[Symmetric Decomposition]\label{SymmDecompdef}Let $\A$ be an AG algebra. 
We define the ideal $C(a)\subset \A^\ast$ by
\begin{equation}
C(a)_i= \m^{\,i}\cap (0:\m^{\,j+1-a-i})/\bigl( \m^{\,i+1}\cap 
(0:\m^{\,j+1-a-i})\big).
\end{equation}
We let $Q(a)_i=C(a)_i/C(a+1)_{i}$.
\end{definition}
Then
\begin{equation}\label{Qaeqn}
Q(a)_i\cong
\frac{\m^{\,i}\cap (0:\m^{\,j+1-a-i})}
{\m^{\,i}\cap (0:\m^{\,j-a-i})+
\m^{\,i+1}\cap (0:\m^{\,j+1-a-i})}.
\end{equation}
Recall from \cite[Theorem 1.5]{I2} or \cite[\S 1.1]{IM}, 
\begin{lemma}[Symmetric Decomposition of $\A^\ast$]\label{symmdecomplem}
Let $\A$ be an Artinian Gorenstein algebra, then we have a filtration
\begin{equation}\label{strataeq}\A^\ast=C_{\A}(0)\supset C_{\A}(1)\supset C_{\A}(2)\supset\cdots  \supset C_{\A}(j-2)\supset C_{\A}(j-1)= 0,
\end{equation}
whose successive quotients $Q_{\A}(a)=C_{\A}(a)/C_{\A}(a+1)$, satisfy, $Q_{\A}(0)$ is a standard-graded Artinian Gorenstein algebra, while each $Q_{\A}(a)$, $a\ge 1$ satisfies 
$\Hom_{\kk}\bigl(Q(a)_i,\kk\bigr)\cong Q(a)_{j-a-i}$, that is, the double dual of $Q(a)$ as graded $\kk$-module is isomorphic to $Q(a)$, after the shift; the center of symmetry of $Q(a)$ is $(j-a)/2$. \footnote{We will show in Example \ref{non-reflexex} that $Q(a)$ is not necessarily reflexive as an $A^\ast$-module.}
\end{lemma}
Letting $H_{\A}(a)$ be the Hilbert function of $Q_{\A}(a)$, we define the \emph{symmetric decomposition} $\mathcal{D}(\A)$ of the Hilbert function $H(\A)$ as the sequence
\begin{equation}\label{decompeqn}
\mathcal{D}(\A)=\bigl(H_{\A}(0),H_{\A}(1),\ldots,H_{\A}(j-2) \bigr).
\end{equation}
This symmetric decomposition has recently been used by several groups, particularly in the study of short or special Artinian Gorenstein algebras, and in the study of the scheme lengths of forms: see \cite{BJMR,ER1,ER2,Je} and the discussion and references in \cite{IM}. Writing $\mathsf{m}$ for the maximal ideal of $R$, we denote by $\Gor(H)\subset \Grass(n,R/\mathsf{m}^n)$, the family (an algebraic set) of AG algebra quotients $I\to R\to \A=R/I$ of $R$ having Hilbert function $H$: thus, a point $p_\A$ of $\Gor(H)$ specifies the ideal $I\subset R$.  We denote by $\Gor(\D)\subset \Gor(H)$ the subfamily of  $\Gor(H)$ parametrizing AG algebras whose associated graded algebras have symmetric decomposition $\D$.\vskip 0.2cm   
Our next basic ingredient is the Jordan type $P_{\ell,A}$ for a pair consisting of an Artinian local algebra $A$ and an element $\ell\in \m_A$. This was introduced in 
\cite[\S 3.5]{H-W} and in \cite[Definition 2.1, Section~2.2]{IMM1}, as a generalization of the much-studied concept of $A$ being weak or strong Lefschetz (see \cite{H-W}).      
\begin{definition}
\label{JTdef}
\cite{IMM1}, \cite[\S 3.5]{H-W} The \emph{Jordan type} $P_\ell$ of an element  $\ell\in\m$ is the partition giving the Jordan block decomposition of the multiplication map $m_\ell: \A\to \A$, $m_\ell (a)=\ell \cdot a$. Since $\ell$ is nilpotent, the eigenvalues of $m_\ell$ are zero. We will say that $(\A,\ell)$ is a \emph{strong Lefschetz pair} if the Jordan type $P_\ell=H(\A)^\vee$, the conjugate partition (switch rows and columns in the Ferrers graph) of $H(\A)$.\footnote{To obtain $H(\A)^\vee$ we reorder the sequence $H(\A)$ to make it non-increasing, and then take the conjugate.}
We say that $\A$ is strong Lefschetz (in \cite{IMM1} this is termed ``strong Lefschetz Jordan type'') if there exists a strong Lefschetz pair $(\A,\ell)$. The \emph{Jordan type} $P_\A$ (the generic Jordan type) of $\A$ is the partition $P_\ell$ for $\ell$ a generic element in $\mathfrak{m}_\A$.\par
When $P_A=(p_1,\ldots, p_t)$ (we will assume $p_1\ge p_2\ge \cdots \ge p_t$) we may write $\A=\oplus \mathcal{L}(i)$, where each $\mathcal{L}(i)\cong {\mathsf{k}}[t]/(t^{p_i})$ is a simple $k[t]$-module  under the action $t( a)= \ell\cdot a$. We call such a submodule of $\A$ an \emph{$\ell$-string} $\mathcal{L}(i)$ of~$\A$, and we will call the representatives of $t^i, i\in [0,p_1-1]$ in $A$ the \emph{beads} of the string. 
\end{definition}
\begin{example} The algebra $\A={\mathsf{k}}\{x,y\}/(x^2+y^3,\, xy)$ of Hilbert function $H(\A)=(1,2,1,1)$ is strong Lefschetz as $\ell=(x+y)$ satisfies $\ell^3=y^3\not=0$ in $\A$, but $\ell^4=0$, so there is a string $\langle 1,\ell,\ell^2,\ell^3\rangle$. Here $\langle x+y^2\rangle$ is a length-one string as $\ell(x+y^2)=0$, implying the Jordan type is $P_\ell=(4,1)=H(A)^\vee$.
\end{example}
We briefly recall the Macaulay duality for an Artinian Gorenstein ring:  the local ring $R=\kk\{x_1,\ldots, x_n\}$ acts on the divided power ring ${S=\kk_{DP}[X_1,\ldots,X_r]}$ by contraction: thus $x_i^{\,k}\circ X_i^{\,j}=X_i^{\,j-k}$ if $j\ge k$ and is $0$ otherwise. For an ideal $I\subset R$ we denote by $I^\perp=\{F\in S\mid h\circ F=0, \forall h\in I\}$; for an element $F\in S$ we denote by $\Ann F\subset R$ the ideal $\Ann F=\{h\in R\mid h\circ F=0\}$. Let $H$ be a length $n$ Gorenstein sequence of codimension $r$, one occurring as the Hilbert function of an Artinian Gorenstein algebra quotient of $R$ of socle degree $j$ ($H_j\not= 0$, $H_i=0$ for $i>j$).  We denote by $S_{\le j}$ the $R$-submodule of $S$ comprised of all elements in $S$ of degree at most $j$, and by $\mathcal E_j$ the set of classes $\overline{F}$ of degree-$j$ elements by the contraction action of units in $R$; that is for $F$, $G$ of degree $j$ in $S_{\le j}$ the class $\overline{F}=\overline {G}$ if there exists a unit $ u\in R^\ast$ such that $G=u\circ F$; we denote by $\mathcal E_{j,H}$ the subset of classes in $\mathcal E_j$ such that $H(R\circ F)=H$ (equivalently, so $H(R/\Ann F)=H$).  Then we have 
\begin{lemma}\cite{Mac},\cite[Lemma 1.1]{I2}\label{Macaulalyduallem} 
The Artinian Gorenstein (AG) algebras $\A=R/I$ of socle degree $j$ correspond 1 to 1 to the set $\mathcal E_{j,H}$ of classes $\overline{F}$ of degree $j$ elements $F\in S_{\le j}$ $\mod$ the contraction action by units of $R$. The class $\overline{F}$ determines the submodule $\A^\vee=R\circ F\subset S$ of Hilbert function~$H$. That is, we define
\begin{equation}
\begin{aligned}
\iota&:\  \Gor(H)\to \mathcal E_{j,H} &\iota( \A)&=\overline{F}, \text { where }R\circ F=I^\perp\\[1ex]
\iota^\prime&: \mathcal E_{j,H}\to \Gor(H),  &\iota^\prime (\overline{F})&=\A_F \text { where }\A_F=R/\Ann F.
\end{aligned}
\end{equation}
Then $\iota$ and $ \iota^\prime=\iota^{-1}$ are inverse isomorphisms.\par
For $F=F_j+F_{j-1}+\cdots$ and $\A=R/\Ann F$, the maximum length graded Artinian Gorenstein quotient of $\A^\ast$ is  $Q_\A(0)=\A^\ast/C_\A(1)$, and $Q_\A(0)=R/\Ann F_j$. 
\end{lemma}\noindent
We term an element $F\in S$ determining an AG algebra $\A=R/I$, $I=\Ann F$, a Macaulay \emph{dual generator} of $\A$. It is a way to view and to construct AG algebras, that we will use frequently. For a more detailed study see \cite{I2,IM,KK}.\par
Next, we show by example that, although the modules $Q(a)$ satisfy a nice reflexive property as $\kk$-vector spaces, they are not reflexive in general as $\A^*$-modules.
\begin{example}\label{non-reflexex}
Let ${R=\kk[x,y,z]}$, consider ${F=X^3YZ+Y^4}$, and take ${A=R/\Ann F}$. Then ${H(A)=(1,3,5,4,3,1)}$, and its symmetric decomposition is 
\[
\D=\bigl(H(0)=(1,3,4,4,3,1),\, H(1)=(0,0,1,0,0)\bigr).
\]
We have that ${\Ann F=(z^2,\,xy^2,\, y^3-x^3z,\, y^2z,\, x^4)}$ and ${A^*=R/(\Ann F)^*}$ has a two-dimesional socle, generated by $y^2$ and $x^3yz$, since $(\Ann F)^*=(z^2,\,xy^2,\, y^3,\, y^2z,\, x^4)$. We can check that $Q(1)=\langle y^2\rangle$, and therefore ${
\Hom_{A^*}\bigl(Q(1),A^*\bigr)=\langle\phi_1,\phi_2\rangle}$, where $\phi_1(y^2)=y^2$ and $\phi_2(y^2)=x^3yz$, since the image of $y^2$ must be in the socle of $A^*$. But now we see that  
the double dual of $Q(1)$ 
has dimension four as a $\kk$-vector space, since a morphism from 
$\Hom_{A^*}\bigl(Q(1),A^*\bigr)$
to $A^*$ is defined by
\[
\phi_1\mapsto a_1y^2+b_1x^3yz, \qquad \phi_2\mapsto a_2y^2+b_2x^3yz.
\]
In particular, we cannot have an isomorphism between $Q(1)$ and its double dual, so $Q(1)$ is not reflexive as an $A^*$-module. 
\end{example}

\subsection{Background and Summary.}
We will assume henceforth that the regular local ring $R={\mathsf{k}}\{x,y\}$ or ${\mathsf{k}}\{x,y,z\}$, according to context -- unless otherwise stated. Given a Gorenstein sequence $H$ (not necessarily symmetric) recall that by $\Gor(H)$ we understand the (reduced) variety whose geometric points are the local Gorenstein quotients $\A=R/I$ having Hilbert function $H$.
It is known, since J. Brian\c{c}on's thesis \cite{Bri} and the first author's Memoir \cite{I1} that for $T$ any O-sequence of codimension two (satisfying Equation \eqref{Tcod2eq}), and for $\sf k$ algebraically closed, the variety  $Z_T$ parametrizing all ideals of Hilbert function $T$ is smooth and irreducible, and is fibred by affine spaces over the smooth projective variety $G_T$ parametrizing graded ideals of Hilbert function $T$ (see Lemma \ref{ZTlem} below). For a codimension two Gorenstein sequence $H$ (see Lemma \ref{ZTlem}.A) the variety $\Gor(H)$ has Zariski closure $Z_H$. Also in codimension two, given $H$, there is a unique symmetric decomposition \cite[Theorem 2.2]{I2}; finally, $\Gor(H)$ is smooth of known dimension \cite{Bri,I1}, fibred by opens in an affine space over $\GGor(H)$ \cite{I1} (see Section~\ref{twotothreesec} and for the fibration also \cite{HH}).
Thus, $\Gor(H)$ is irreducible in codimension two. \par
In codimension three S. Diesel showed that the \emph{graded} Gorenstein algebras $\GGor(H)$ of Hilbert function $H$ form an irreducible family of known dimension \cite{Di}: to prove this she used the D.~Buchsbaum - D.~Eisenbud Pfaffian structure theorem \cite{BuEi}. She also showed results about the poset of generator-relation degrees corresponding to a given Hilbert function $H$ (for a summary, see \cite[Theorem~5.25]{IK}). In codimension four, the family  $\GGor(H)$ of graded Gorenstein algebras $\A=R/I$ of Hilbert function $H$ may itself have several irreducible components: this was shown in \cite[Theorem 4.9]{IS} for the infinite class of graded Gorenstein sequences beginning $H=(1,4,7,h,b,\ldots)$, $8\le h\le 10$, $b\le 3h-17$. A general method here in codimension four is interpreting the ideal $(I_{\le 3})$ generated by the initial portion of $I$ as the ideal of a curve in $\mathbb P^3$.  Then some of the complicated component structure of the Hilbert scheme of curves on $\mathbb P^3$ can be shown to transfer to give irreducible components of $\GGor(H)$ (see also \cite{K} and \cite[Chapter 6]{IK}).  \vskip 0.2cm
\subsubsection{Summary.}
Section 2 presents tools we will use, as well as examples and questions.
Sections \ref{specJTsec} and \ref{semicontsec} report the well-known results of how Jordan types (Section \ref{specJTsec}) and symmetric decompositions $\mathcal D$ (Section \ref{semicontsec}) may specialize.  Section \ref{twotothreesec} gives a brief summary of parametrization results about graded Gorenstein algebras in codimension two, for the case - normal for us - when $Q(0)$ has codimension two.  Section \ref{JTandSDsec} asks about the relation between Jordan bases for an AG algebra $A$ and $A^\ast$, and the symmetric decomposition of $A^\ast$: can we find a Jordan basis, or even a pre-Jordan basis for $A$ (not requiring each string to end in zero, Definition \ref{preJordandef}) that induces a basis for each $Q_A(a)$. In general we cannot (Example \ref{2to3ex}), and this is an insurmountable obstruction to directly defining a notion of Jordan order type for non-graded AG algebras that would generalize the symmetric Jordan degree type (specify the initial degrees of strings)  for graded AG algebras introduced implicitly in \cite{HW} and explicitly in \cite[\S 2.6]{IMM1} and \cite{CGo}.\par
 In Section \ref{JTA,A*sec} we study in codimension three the Jordan types of $\A$, of its associated graded algebra $\A^\ast$ and of the subquotients of $\A^\ast$ that we term the symmetric decomposition of $\A^\ast$ (Definition \ref{SymmDecompdef}).
There is, as in the graded case, an inequality between the Jordan type of $\A$ and the conjugate partition of the Hilbert function $H(\A)$ (Lemma \ref{SLupperbdlem} below from \cite[Theorem~2.5]{IMM1}). We give several examples where $P_A>P_{A^\ast}$, including Example \ref{symHilbex} where $H(A)$ is symmetric.  We introduce in Definition \ref{def:relativeLef} further invariants, a contiguous partition $P_c(H)$ associated to $H$, and related invariants  $P_\D$ and $P_c(\mathcal{D})$ relevent when $H$ or a component $H(a)$ of $\D$ is non-unimodal; and we give in Lemma \ref{contiglem} an inequality between $P_c(H)$ and $P_{A,\ell}$ when $H(A)=H$. \par In Section \ref{2compsec}
 our Example~\ref{AA*diffex} is our first using the previous tools to show that $\Gor(H)$ - here $H=(1,3,4,4,4,2,1)$ - may have several irreducible components. There, a single stratum $\Gor(\D_2)$ has several irreducible components (Claim \ref{Dtwocompclaim}).  This example shows also that a concantenated Jordan type associated to the symmetric decomposition of $\A^\ast$ may be larger in the dominance partial order than that of $\A^\ast$, a surprise to us. In Example~\ref{thirdAastex} we give a simple case illustrating for $H=(1,3,5,4,4,2,1)$ our method of showing that $\Gor(H)$ may have several irreducible components in codimension three.\par
 In Section~\ref{exoticsec} we give a concise presentation of the theory of exotic summands of the Macaulay dual generator $F$, from \cite[\S 3.2]{BJMR}; in Lemma \ref{exoticlem} we determine the number of  exotic parameters of $\Gor(H)$ when $H\bigl(Q(0)\bigr)$ has codimension two, and $H(\A)$ has codimension three. This will help us determine the dimensions of some of the irreducible components of $\Gor(H)$ in the following sections. The concept of relatively compressed modification in Section \ref{rcsec} from \cite{I2,IM} helps in giving the structure of some of the components $\Gor(\D)$ we will see later.
\par In Section \ref{applysec} we study the irreducible components of several $ \Gor(H)$ in codimension three.\par
  In contrast to the case  of symmetric $H$ in codimension three, here the variety $\Gor(H)$ parametrizing all AG algebras of a non-symmetric Hilbert function $H$ may be reducible. This was shown by the first author for the Hilbert function $H=(1,3,3,2,1,1)$; the proof used the semicontinuity of symmetric decomposition (Lemma \ref{trianglelem}) and a dimension calculation \cite[Theorem 4.2]{I2}. Our main results in Section~\ref{3compsec} introduce a new method for showing the reducibility of certain $\Gor(H)$ in codimension three, by comparing the Jordan type of $A$ under dominance partial order, and a natural partial order on symmetric decompositions.  We use this method to show that there are three irreducible components for $H=(1,3,4,4,3,2,1)$, corresponding to the three occurring symmetric decompositions:
 \begin{thm} (Theorem \ref{2.1thm}) The Gorenstein sequence $H=(1,3,4,4,3,2,1)$ has  three symmetric decompositions,
 \begin{align*} 
 \D_1&=\bigl(H(0)=(1,2,3,4,3,2,1),\, H(1)=H(2)=0,\, H(3)=(0,1,1,0)\bigr)\\
 \D_2&=\bigl(H(0)=(1,2,3,3,3,2,1),\, H(1)=0,\, H(2)=(0,1,1,1,0)\bigr)\\
 \D_3&=\bigl(H(0)=(1,2,2,2,2,2,1),\, H(1)=(0,1,2,2,1,0)\bigr).
 \end{align*}
 The generic Jordan type of these algebras, is, respectively  $P_1= (7,5,3,2,1)$, $P_2= (7,5,3,3)$, and $P_3
 = (7,5,4,2)$. These symmetric decompositions stratify $\Gor(H)$ into three irreducible components of dimensions, respectively, $29$, $34$ and $28$ for $\D_1$, $\D_2$ and $\D_3$.
 \end{thm}
We extend this result to the infinite set of Gorenstein sequences $H=(1,3,4^k,3,2,1)$, $k\ge 2$ in Theorem \ref{3compthm}. In Section \ref{infsec}, Example \ref{2.4ex} and Theorem \ref{infinitethm}, we exhibit a doubly infinite sequence of families $\Gor(H)$ in codimension three, each $H$ having having at least two symmetric decompositions, each of which determines an irreducible  component of $\Gor(H)$. Our method of proof uses the dominance partial order on Jordan type (Lemma \ref{dominancelem}) and the semicontinuity of symmetric decomposition (Lemma \ref{trianglelem}, \cite[\S 4]{I1}). Perhaps surprisingly, these work in opposite directions. \par
 There are obvious questions concerning the growth of the number of irreducible components of $\Gor(H)$ as function of the codimension, socle degree and Sperner number (maximum height) of $H$, which we plan to address in a subsequent paper.   \par
 In Section \ref{realizablesec} we discuss whether the natural conditions on a potential symmetric decomposition of $H$ implied by Lemma \ref{symmdecomplem} are sufficient: is any such potential decomposition \emph{realizable}, that is, does it actually occur for some AG algebra $\A$?  An example with answer ``no'' was given for $H=(1,4,3,4,3,1)$, $\D=\bigl(H(0)=(1,3,3,3,3,1),\, H(1)=(0,1,0,1,0)\bigr)$  in \cite[p.99]{I2}; S. Masuti and M. Rossi have shown that the answer is ``Yes'' for socle degree at most four \cite{MR}. Very recently S. Masuti has obtained partial results in socle degree five, suggesting that there are only a very few non-realizable decompositions \cite{Mas}. We show a second main result  (Theorem \ref{non-realizablethm}):
 \begin{thm}
Let $R=\kk\{x_1,\ldots, x_r\}$, $S=\kk_{DP}[X_1,\ldots X_r]$, $R'=R\{w\}$, $S'=S_{DP}[W]$. Let $B$ be a graded Gorenstein algebra of Hilbert function $H(0)$ of socle degree $j$, and assume that $B=R/I_B$ where $ I_B=\Ann f\cap R$ for an element $f\in S_j$. Assume further that $j\ge 5$ and that $I_B$ is generated in degrees less or equal $j-3$.\footnote{In effect, $j\ge 8$}. Let $H(1)=(0,1,0,\ldots, 0,1_{j-2},0)$, and assume that $H=H(0)+H(1)$ satisfies the Macaulay conditions.  Then $\D=\bigl(H(0),H(1)\bigr)$ is not realizable as the (complete) decomposition sequence of an AG algebra quotient $A$ of $R'$ having $Q(0)$ isomorphic to $B$.
 \end{thm} 
 \par The question of non-ubiquity asks whether there is a decomposition $\D(\A): H=\sum_0^k H(a)$  such that some $H'=\sum_0^{k'}H(a)$ is not realizable as $\mathcal{D}(B)$ for a value $k' <k$. In \cite{IM} we showed a partial non-ubiquity result where we imposed further conditions on $B$. Our examples here, however, do not imply a non-ubiquity result, and the general question of non-ubiquity remains open.\par
  We discuss related questions and pose further problems in Section \ref{probsec}.
\section{Preliminary results.}
We set out ingredients we will use to work with deformations of algebras.
\subsection{Specialization of Jordan type.}\label{specJTsec}
We first recall the dominance partial order on partitions of $n$. In Lemma \ref{SLupperbdlem} (see \cite[Theorem 2.5]{IMM1}) we remind that the conjugate $H(A)^\vee$ (Definition \ref{JTdef}) to $H(A)$ is an upper bound for the Jordan type of an element $\ell\in \m_A$. 
\begin{definition}\label{dominancedef}[{Dominance} order on partitions of $n$.] Let $ P=(p_1,p_2,\ldots, p_s)$, $p_1\ge p_2\ge \cdots\ge p_s $ and $P^\prime=(p_1^\prime, p_2^\prime,\ldots p_{s^\prime}^\prime)$ with $p^\prime_1\ge p^\prime_2\ge \cdots\ge p^\prime_{s^\prime}$.  Then we say $P\le P^\prime$ in the dominance partial order if for each $i\in [1,\max\{s,s^\prime\}]$ we have
\begin{equation} 
\sum_{k=1}^i p_i\le \sum_{k=1}^i p^\prime_i.
\end{equation}
\end{definition}
The following results are well-known. We thank Steven Kleiman for the two references
D. Mumford \cite[Prop. 8, p. 44]{Mumf} (with an unnecessary hypothesis of $W$ irreducible) and
R.~Hartshorne \cite[Exercise 5.8(c), p. 125]{Hart} for the first Lemma.
\begin{lemma}[Flat family]\label{flatlem} Suppose $\mathsf{k}$ is algebraically closed, that $W$ is a reduced variety and let $
\pi: M \to W$ be an algebraic family, whose fibres are Artinian schemes.   Then  $M $ is a flat family if and only if the length of the fibres $M_w=\pi^{-1}(w)$, $w\in W$  is constant.
\end{lemma}
\begin{lemma}\label{dominancelem}
Suppose $\mathsf{k}$ is algebraically closed. Let $\AAA(w)$, $w\in W$ be a flat family of Artinian algebras over a reduced, irreducible parameter space $W$. Then there is a generic Jordan type $P_W$: it is the generic Jordan type of $\AAA(w)$ that occurs for an open dense set of $w\in W$. Let $w_0\in W$. Then the generic Jordan type $P_0$ of $\AAA(w_0)$, satisfies $P_W\ge P_0$ in the dominance partial order of Definition \ref{dominancedef}.
\end{lemma} 
\begin{proof}
This is straightforward to show from the semicontinuity of the rank of $\ell(t)^i$, see \cite[Thm. 6.2.5]{CM} or \cite[Corollary 2.44]{IMM1}.
\end{proof}\par
We have from \cite[Theorem 2.5]{IMM1}
\begin{lemma}\label{SLupperbdlem} Let $\A=(\A,\mathfrak{m},{\mathsf{k}})$ be a local Artinian algebra over $\mathsf{k}$ 
and let $M\subset \A^k$ be an $A$-module, with Hilbert function $H(M)$. For any $\ell\in\mathfrak{m}_\A$, its Jordan type satisfies
\begin{equation}
\label{eq:JTA}
P_{\ell,M}\leq H(M)^\vee
\end{equation}
in the dominance partial order.
\end{lemma}
Note that the associated graded algebra $\A^\ast$ of an Artinian algebra $\A$ is a specialization
of $\A$, hence $P_\A\ge P_{\A^\ast}$. 
\begin{lemma}[Dominance]
\label{specializelem}  
Suppose $\mathsf{k}$ is algebraically closed, and let $\AAA(w)$, $w\in W$ be a flat family of  finite AG algebras over a reduced, irreducible parameter space $W$. 
 Denote by $P^\ast_W$ the generic Jordan type of $\AAA^\ast (w)$ for $w$ in an open dense subset of $W$; and by $P^\ast_0$ the generic Jordan type
of $\A_0^\ast$ where $\A_0=\AAA(w_0)$. 
 Then $P^\ast_W\ge P^\ast_0$ in the dominance partial order of Definition~\ref{dominancedef}. 
\end{lemma}
\begin{proof} By Lemma \ref{flatlem} the family $\AAA^\ast(w)$, $w\in W$ is also flat, as the fibres have constant length.  By examining the behavior of $\m_{\AAA(w)}^{\,i}$ we may conclude that the  the family $\AAA^\ast (w)$, $w\in W$ indeed specializes to $A_0^\ast$.\par
This also follows from the principle that localization is a flat functor, and $A\to A^\ast$ is a localization. The conclusion now follows from Lemma \ref{dominancelem}.\par
\end{proof}\par
\subsection{Specialization of symmetric decomposition.}\label{semicontsec}
We recall from \cite[\S 4]{I2} that there is a partial order $\le$ on symmetric decompositions of a fixed Hilbert function $H$ such that $\overline{\Gor(\D)}\cap \Gor(\D^\prime)\not=\emptyset \Rightarrow \D\le \D^\prime$, that is, the symmetric decompositions $\D$ of the Hilbert function $H$ satisfy a semicontinuity property consistent with the partial order (Lemma~\ref{trianglelem}). This result follows from the semicontinuity of $\dim \bigl(\mathfrak{m}^i\cap (0:\mathfrak{m}^b)\bigr)$ when the Hilberrt function $H(\A)$ is fixed: fixing the Hilbert function fixes $\dim_{\sf k} \m^i$ and by duality fixes $\dim_{\sf k}(0:\m^b)$, so the intersection dimension is semicontinuous.\par
Given a symmetric decomposition $\mathcal{D}$ of a fixed AG Hilbert function $H$, we write
\begin{equation}
H\bigl(\mathfrak{m}^i/\bigl(\mathfrak{m}^i\cap (0:\mathfrak{m}^b)\bigr)\bigr)=(0,\ldots ,0, n_{i,b},n_{i+1,b},\ldots, n_{j,b}).
\end{equation}
We set $N_{i,b}=\dim_\kk \bigl(\mathfrak{m}^i/\bigl(\mathfrak{m}^i\cap (0:\mathfrak{m}^b)\bigr)\bigr)=\sum_{k=i}^jn_{k,b}$. We have that
\begin{equation}
N_{i,b}=\sum_{a=0}^{j-b-i}\sum_{k=i}^{j-b-a}H(a)_k,
\end{equation}
the sum of the entries of $\D$ (as usually arranged) in a triangle bounded on the left by the leading degree $i$ and on the right by the rising diagonal ending at $H(0)_{j-b}$.\par

\begin{lemma}
\label{trianglelem}
\cite[Lemmas 1.13, 4.1A,B]{I2} 
Let $\AAA(w)$, $w\in W$ be a (flat) family of  Artinian Gorenstein algebras of fixed Hilbert function $H$, parametrized by a scheme $W$, and fix integers $i,b,m$. 
The condition $N_{i,b}\bigl(\AAA(w)\bigr)\le m$ is a closed condition on the family $\AAA(w)$, $w\in W$. 
\end{lemma}
\begin{proof} Fixing $H$ fixes the vector space dimensions of $\mathfrak{m}^i$ and $(0:\mathfrak{m}^b)$ (by duality  $\dim_{\mathsf{k}}{(0:\mathfrak{m}^b)}=\dim_{\mathsf{k}}\A-\dim_{\mathsf{k}}\mathfrak{m}^b$). Thus $N_{i,b}\le m$ expresses the condition that the two subspaces $\m_{\AAA(w)}^{\,i}$ and ${\bigl(0:\m_{\AAA(w)}^{\,b}\bigr)}$ of $\AAA(w)$, each having a fixed dimension, intersect in dimension at least ${(\dim_{\kk}\mathfrak{m}^i-m)}$. This is a closed condition on $w\in W$.
\end{proof}\par
See Figure \ref{symdecfig}, where $N_{2,3}$ for $\D_1$ is $8$, for $\D_2$ is $7$ and for $\D_3$ is 6, implying there can be no specializations from $\Gor(\D_i)$ to $\Gor({\D_j})$ for $i>j$, when $H=(1,3,4,4,3,2,1)$. We define a partial order on symmetric decompositions, 
\begin{equation}\label{PODeq}
\D_1\ge \D_2 \text { if }  N_{i,b}\bigl(\D_1)\bigr)\ge N_{i,b}\bigl(\D_2)\bigr) \text{ for each pair $(i,b)$}.
\end{equation}
\begin{corollary}[$H_\A(0)$ as obstruction to deformation in $\Gor(H)$]
\label{D0cor}
Suppose that $\D_s$, $\D_t$ are two symmetric decompositions of $H$ such that  $\mathcal{D}_s (0)_i< \D_t(0)_i$, Then $\overline{\Gor({\D_s})}\cap \Gor(\D_t)=\emptyset.$ 
\end{corollary}
\begin{proof} Take $b=j-i$ in the Lemma \ref{trianglelem}
\end{proof}\par
In the Figure \ref{symdecfig}, the  Corollary for $i=3$ shows again that there is no specialization from $\Gor(\D_i)$ to $\Gor(\D_j)$ for $i>j$, with $H=(1,3,4,4,3,2,1)$.
\par

\begin{figure}
\center{$N_{2,3}$:}\vskip 0.3cm
$\begin{array}{ccccccc}
1&2&\color{red}{3}&\color{red}{4}&3&2&1\\
&1&\color{red}{1}&&&&\\
&&\mathcal{D}_1&&&
\end{array}$\qquad
$\begin{array}{ccccccc}
1&2&\color{red}{3}&\color{red}{3}&3&2&1\\
&1&\color{red}{1}&1&&&\\
&&\mathcal{D}_2&&&
\end{array}$\vskip 0.4cm\par
$\begin{array}{ccccccc}
1&2&\color{red}{2}&\color{red}{2}&2&2&1\\
&1&\color{red}{2}&1&&&\\
&&\mathcal{D}_3&&&
\end{array}$\vskip 0.4cm
We have $\overline{\Gor(\D_3)}\cap\bigl(\Gor(\D_2)\cup \Gor(\D_1)\bigr)=\emptyset$,\vskip 0.2cm\par
\qquad $\overline{\Gor(\D_2)}\cap\Gor(\D_1)=\emptyset$.\vskip 0.2cm
\caption{Obstruction to specialization of symmetric decompositions, $H=(1,3,4,4,3,2,1)$ (Lemma \ref{trianglelem}).}\label{symdecfig}
\end{figure}\vskip 0.2cm
As mentioned, there is a unique symmetric decomposition of $H$ when $H=(1,2,\ldots )$ is a Gorenstein sequence of codimension two (so the Hilbert function of a complete intersection).  Not much is known about the poset $\mathfrak D(H)$ of symmetric decompositions of a given Gorenstein sequence $H$ of codimension at least three:  see Question~\ref{posetDques}. We do not even know which (non-graded) Gorenstein sequences in codimension three occur for complete intersections!
\subsection{When the graded quotient  $Q_\A(0)$ of $\A$ has codimension two.}\label{twotothreesec}
\par We first record some well-known facts about parametrizing  AG algebras of codimension two (Lemmas \ref{ZTlem} and \ref{cod2lem}). We have already stated the results we will need about specialization of Jordan type (Lemma \ref{dominancelem}) and the semicontinuity of symmetric decompositions (Lemma~\ref{trianglelem}).
Let
\begin{equation}\label{Tcod2eq}
T=(1,2,\ldots, d,t_d,t_{d+1},\ldots ,t_j,0), \text { where } d\ge t_d\ge \cdots \ge t_j>0,
\end{equation} be a sequence possible for the Hilbert function of an Artinian quotient of $R={\mathsf{k}}\{x,y\}$ (called a codimension two O-sequence) and let $e_i=t_{i-1}-t_i, n=|T|=\sum t_i$. Here $d$ is the order of ideals defining algebras of Hilbert function $T$: $\mathfrak{m}^d\supset I$ but $\mathfrak{m}^{d+1}\not\supset I$. Recall that a codimension two Artinian Gorenstein algebra is a complete intersection. Then we have (for the case $Z_T$ in (B.) see also \cite{Bri}),
\begin{lemma}See \cite[Theorem 2.12]{I1}\label{ZTlem}
\begin{enumerate}[A.]
\item \cite{Mac1} The sequence $T$ of Equation \eqref{Tcod2eq} can occur for an Artinian Gorenstein quotient of the local ring $R={\mathsf{k}}\{x,y\}$ if and only if $e_i\le 1$ for each $i\ge d$.
\item Let $Z_T$ parametrize the Artinian algebra quotients $A=R/I$ having Hilbert function $T$, and $G_T$ parametrize the graded quotients having Hilbert function $T$.
Then $Z_T$, $G_T$ are smooth irreducible varieties, $\pi: Z_T\to G_T: \pi(\A)=\Gr_{\m}(\A)=\A^\ast$ is a bundle with affine space fibres over the projective variety $G_T$; the natural inclusion map $G_T\to Z_T$ is a zero-section of $\pi$. Each  of $Z_T$, $G_T$ has a cover by opens in affine spaces. We have 
\begin{align}
\dim Z_T&=n-\sum_{i\ge d} e_i(e_i+1)/2\notag\\
&=n-d-\sum e_i(e_i-1)/2;\label{dimZteq}\\
\dim G_T&=\sum_{i\ge d} (e_i+1)(e_{i+1}).\label{dimGteq}
\end{align}
\end{enumerate}
\end{lemma}

\begin{lemma}[Codimension two implies strong Lefschetz]\label{cod2lem}
Let $\A$ be a local Artinian algebra of codimension two, thus $A={\mathsf{k}}\{x,y\}/I$ and assume that $\cha \mathsf{k}$ is either zero or greater than the socle degree of $\A$. Then $A$ has a linear element $\ell$ that is strong Lefschetz.
\end{lemma}
 \begin{proof} This arises from standard basis theory of  J.~Brian\c{c}on \cite{Bri} (in characteristic zero), or its analogue in characteristic $p$ \cite{I1}. See also \cite[Lemma~2.14]{IMM1} or \cite[Theorem 2.9]{I2}). 
\end{proof}\par
 We can use Lemma \ref{cod2lem} as part of ruling out strong Lefschetz for  AG algebras $\A\in \Gor(\D)$ for certain decompositions where $\A$ has codimension three but $\A(0)$ has codimension two. See Example~\ref{2to3ex}. 
 \subsection{Jordan type and symmetric decomposition.}\label{JTandSDsec}
We are concerned here with the interplay between Jordan type and the symmetric decomposition $\mathcal{D}(\A)$ of the Hilbert function $H(\A)$. Here, we have many questions, including whether and how we can define a ``Jordan order type'' for non-homogeneous AG algebras. Examples~\ref{ChrisMex} and \ref{2to3ex} illustrate some of the issues.\footnote{In work in progress with Johanna Steinmeyer \cite{IMS} we have concluded that although there is a notion of sequential Jordan type for non-homogeneous AG algebras, our several tries at a ``Jordan order type'' have not succeeded, see Question \ref{Jordanordertypeques} .}
\begin{definition}\label{preJordandef}
A \emph{pre-Jordan basis} $\mathfrak{B}$ of $\A$ with respect to a non-unit $\ell\in \A$, of Jordan type $P_\ell=(p_1,p_2,\ldots, p_s)$ is a basis of $\A$ as a vector space over $\kk$ 
\begin{equation}\label{preJordaneq}
\mathfrak{B}=(\ell^k b_i : 1\le i\le s,\, 0\le k<p_i )
\end{equation}
where $\ell^{p_i}b_i\in {\mathsf{k}}[\ell]\cdot \langle b_1,b_2,\ldots ,b_{i-1}\rangle$. Recall that if for each $i\in [1,s]$, $\ell^{p_i}b_i=0$ then $\mathfrak{B}$ is a Jordan basis for $(\A,\ell)$. We will call a sequence $(b_i,\ell b_i,\ldots, \ell^{p_i-1}b_i)$ as in Equation \ref{preJordaneq} a \emph{pre-string} for the pre-Jordan basis, whose beads are  $\ell^a b_i, a\in [0,p_i-1]$.
\end{definition}\vskip 0.2cm\par\noindent
\begin{remark}\label{multrem}
We will have use for the fact, for $A$ an AG algebra (not necessarily graded) and  $\ell, h\in A$,
\begin{equation}\label{baseeq}
\text {If }\ell\in \m^a \text { and } h\in \m^b\cap (0:\m^c) \text { then }\ell\cdot h\in \m^{a+b}\cap (0:\m^{c-a}).
\end{equation}
Also if $\ell\in (0:\m^e)$ then $\ell h\in (0:m^{e-b})$.\par
As a consequence, if $h\in A$ represents an element of $Q(a)_i$ and $\ell h$ represents an element of $Q(b)_k$ where $b<a$ then $k\ge i+2$.
\end{remark}
Our next example $A$ has a Jordan basis compatible with the $Q(a)$ decomposition.
\begin{example}\label{preJordanex}
Let $R={\kk}\{x,y\}$, take $f = X^{[4]}Y + Y^{[4]}$, and let $\A=R/\Ann f$. Then
$\Ann f=(xy^2,\, y^3-x^4)$, the Hilbert function is $H(\A)=(1,2,3,2,2,1)$, and its symmetric decomposition is
\begin{equation*}
\mathcal{D}=\bigl(H(0)=(1,2,2,2,2,1),\, H(1)=(0,0,1,0,0)\bigr).
\end{equation*}
Here $P_{\A,x}=(5,5,1)$: the following strings give a Jordan basis for $(\A,x)$ whose elements can be partitioned into bases for each $Q(a)_i$, note that $x^5=xy^2=0$ in $\A$.
\begin{center}
{\renewcommand{\arraystretch}{1.1}
\begin{tabular}{|c|l|}
\hline
$Q(0)$ & 
$
\begin{matrix}
1 &x &x^2 &x^3 &x^4 &&\\
 &y &xy &x^2y &x^3y&x^4y
\end{matrix}
$\\
\hline
$Q(1)$ & 
$
\begin{matrix}
\phantom{1} & \phantom{x} &y^2
\end{matrix}
$\\
\hline
\end{tabular}}
\end{center}
Here $P_{\A,\ell}=(6,4,1)$ for $\ell=x+y$. The following is a pre-Jordan basis for $(\A,\ell)$ again whose elements can be partitioned into bases for each $Q(a)_i$
\begin{center}
{\renewcommand{\arraystretch}{1.1}
\begin{tabular}{|c|l|}
\hline
$Q(0)$ & 
$
\begin{matrix}
1 &\ell &\ell^2 &\ell^3 &\ell^4 &\ell^5\\
 &y &\ell y &\ell^2y &\ell^3y&
\end{matrix}
$\\
\hline
$Q(1)$ & 
$
\begin{matrix}
\phantom{1} & \phantom{y} &y^2
\end{matrix}
$\\
\hline
\end{tabular}}
\end{center}
It is not a Jordan basis, because ${\ell\cdot\ell^3y=x^4y\ne0}$ and ${\ell y^2=x^4\ne0}$. However, we can change this basis and get a Jordan basis that agrees with the $Q(a)$ decompositon of $A^\ast$:
\begin{center}
{\renewcommand{\arraystretch}{1.3}
\begin{tabular}{|c|l|}
\hline
$Q(0)$ & 
$
\begin{matrix}
1 &\ell &\ell^2 &\ell^3 &\ell^4 &\ell^5\\
 &y-\tfrac{1}{5}\ell &\ell y-\tfrac{1}{5}\ell^2 &\ell^2y-\tfrac{1}{5}\ell^3 &\ell^3y-\tfrac{1}{5}\ell^4&
\end{matrix}
$\\
\hline
$Q(1)$ & 
$
\begin{matrix}
\phantom{1} & \phantom{y-\tfrac{1}{5}\ell} &y^2-(x^3-yx^2)
\end{matrix}
$\\
\hline
\end{tabular}}
\end{center}

Here $(0:\m^3)=(x^3,x^2y,y^2)$ so $h=x^3-yx^2=x^2(x-y)\in \m^3\cap(0:\m^3)$.
Now $(x+y)(x^2)(x-y)=x^2(x^2-y^2)=x^4$, so we have $\ell\cdot  (x^3-h)=0$.\par
It is easy to see that the linear form $y$ has a Jordan basis for $P_y=(5,2,2,2)$ since $xy^2=0$.
Thus every linear form has a Jordan basis compatible with the $Q(a)$ decomposition. As we shall see later in Examples \ref{ChrisMex} and \ref{2to3ex} this in general does not occur.
\end{example}
\begin{question} How does the symmetric decomposition $\D(\A)$ restrict the Jordan type of $\A$? In particular can the strong Lefschetz property depend on the decomposition $\D(\A)$? 
\end{question}\noindent
{\bf Answer.} The answer to the second question is ``Yes'', see Theorem \ref{2.1thm} where only the third decomposition is compatible with strong Lefschetz. We develop partial answers to the first question in many of the results and examples of this paper.\vskip 0.2cm
We have the following counterexample due to Chris McDaniel, of a non-homogeneous Artinian Gorenstein algebra $A$ for which there is no Jordan basis consistent with $H(A)$ so, in particular, no Jordan basis consistent with the $Q(a)$ decomposition.
\begin{example}(C. McDaniel, private communication)\label{ChrisMex} Let $ F=XY^3+X^2Y\in k_{DP}[X,Y]$, $A={\kk}\{x,y\}/I$, $I=\Ann F=(x^2-xy^2, y^4)$. Note
$x^3=(x+y^2)(x^2-xy^2)+xy^4\in I$. Then  $H(A)=(1,2,2,2,1)$ and it can be readily shown that $A$ has no Jordan basis consistent with $H(A)$.
Note that since $A$ has a symmetric Hilbert function, the $Q(a)$ decomposition of $A^\ast$ has a single component $Q_A(0)=A^\ast$.
Thus, $A$ does not have a basis consistent with the $Q(a)$ decomposition; but $A^\ast$, being graded, does have a Jordan basis consistent with its Hilbert function.
\end{example}
We also have a result of the second author that every Artinian algebra has a pre-Jordan basis consistent with the Hilbert function \cite{Mar}.

The next Example \ref{2to3ex} is of a codimension three AG algebra that for a particular linear form does not have a pre-Jordan basis that is consistent with the $Q(a)$ decomposition. This is an apparent complication in an attempt to define a notion of Jordan order type for non-graded AG algebras (Question \ref{Jordanordertypeques}). \begin{example}\label{2to3ex}
Let $R=\mathsf{k}\{x,y,z\}$, take $f = X^4Y^4 + X^3Y^3Z + Z^5$ and let $A=R/\Ann f$. Then (from Macaulay) $\Ann f=(yz^2,xz^2, y^4z, x^4z, y^5, xy^4-y^3z, x^4y-x^3z, x^5, x^3y^3-x^2y^2z-z^4)$, $H(A)=(1,3,6,8,7,5,3,2,1)$ of length $36$. Then we have $Q(0)=R/(x^5,y^5,z)$,  $Q(2)\cong  z\cdot  R/(x^3,y^3,z^2)$ and $Q(3)=\langle z^2,z^3\rangle$.
\begin{equation*}
\mathcal{D}=\bigl(H(0)=(1,2,3,4,5,4,3,2,1),\, H(2)=(0,1,2,3,2,1,0),\, H(3)=(0,0,1,1,0,0)\bigr).
\end{equation*}
We wish to show that the multiplication by $\ell = z\in A$ does not admit a pre-Jordan basis {inducing} a basis of the $Q(a)$ decomposition.
Suppose by way of contradiction that there is a pre-Jordan basis $B$ that {can be partitioned into subsets} inducing a basis for each $Q(a)_i$. Since ${z^5=x^4y^4\ne0}$ in $A$, but ${z^6=0}$, there is a string of length $6$ in $B$. Also, we have ${z^5\m=0}$, so this string must start in order $0$, and we can scale it such that it starts in ${1+b}$, with ${b\in\m}$. Its second element is ${z+zb}$ and has order $1$. On the other hand, ${(z+zb)\circ f=X^3Y^3+Z^4+b\circ(X^3Y^3+Z^4)}$. Since ${b\circ(X^3Y^3+Z^4)}$ is a polynomial of degree at most $5$, it cannot cancel the term $X^3Y^3$, so ${(z+zb)\circ f}$ has degree $6$, which tells us that ${z+zb\in(0:\m^7)\setminus(0:\m^6)}$. Thus, ${\m\cap(0:\m^7)}$ is the smallest intersection of powers of the maximal ideal by Loewy ideals that ${z+zb}$ belongs to: so $z+zb$ must represent a nonzero class in 
\[
Q(1)_1=\frac{\m\cap(0:\m^7)}{\m^2\cap(0:\m^7)+\m\cap(0:\m^6)}
  =\frac{(0:\m^7)}{\m^2\cap(0:\m^7)+(0:\m^6)}.
\]
Nonetheless, ${z+zb=(z-xy)+(xy+zb)}$, with ${z-xy\in(0:\m^6)}$, because $(z-xy)\circ f=-X^2Y^2Z+Z^4$, and ${xy+zb\in\m^2\cap(0:\m^7)}$. So ${z+zb}$ represents the zero class in $Q(1)_1$, contradicting the assumption that $B$ is a pre-Jordan basis for $A$.
\vskip 0.2cm\par\noindent
{\bf Some comments.}\par
This argument relied on the fact that there is an element that represents a non-zero class in a module $Q(a)_i$ (here ${1+b}$ in $Q(0)_0$) but when we multiply by $\ell=z$ we get an element that represents zero in all modules $Q(a)_i$ where it fits. We know that we can always find a pre-Jordan basis compatible with the Hilbert function -- a weaker condition -- and the corresponding modules do not have this behavior: any non-zero element of $A$ represents a non-zero class in the quotient $\m^i/\m^{i+1}$, where $i$ is its order.
\par
Now take $h=z-xy$, representing a non-zero element in $Q(2)_1$. Then for $\ell=z$ we have $\ell \cdot h$ is zero in all $Q(a)_i$ where it would fit: the smallest intersection of powers of the maximal ideal by Loewy ideals that it belongs to is $\m^2\cap(0:\m^5)$, making it a good candidate for $Q(2)_2$. However, we can write $zh$ as $(zh+x^2y^2)-(x^2y^2)$, with ${zh+x^2y^2\in\m^2\cap(0:\m^4)}$ and $x^2y^2\in\m^3\cap(0:\m^5)$, showing it represents the zero class in $Q(2)_2$. Nonetheless, $z^2h$ represents a non-zero class in $Q(3)_3$ going to a higher $a=3$, and $z^3h$ represents a non-zero class in $Q(2)_5$ going back to $a=2$, but jumping by two in order, as expected.\par
Note that a generic enough $\ell=x+y+z$ has Jordan type $(9,7,5,5,3,3,2,1,1)$ which is exactly that expected from concatenating the SLJT of the $Q(a)$ decomposition (Macaulay computation). 
However, $\ell=x+y+z$ is not a SL element for $A$ as  $H^\vee=(9,7,6,4,4,3,2,1)$.
\end{example}

\begin{definition}\label{orderdef} Given an element $f\in S$, by the \emph{order} $\mathfrak{o}(g)$ of $g=\varrho\circ f$ we mean the highest power of the maximum ideal ${\maxR}$ of $R$ such that ${g\in {\maxR}^{\mathfrak{o}(g)}\circ f}$, but is not in ${{\maxR}^{\mathfrak{o}(g)+1}\circ f}$. The \emph{degree} of $g$ is $k$ if $g\in S_{\le k}$ but not in $S_{<k}$.
\end{definition}
\par
\begin{remark}\label{exoticbrem}(Exotic summands of $f$)
Let $H$ be a Gorenstein sequence of socle degree $j$ and $\D=\bigl(H(0),\ldots,\, H(j-2)\bigr)$ a symmetric decomposition of $H$.  Let $n_i=n_i(\D)=\sum_{u=0}^i H(u)_1$. 
 Following \cite[Definition~3]{BJMR} we say that $f\in \mathcal S$ determining $A=R/\Ann f\in \Gor(\D)$ is in \emph{standard form} if after a linear change of variables in $\mathcal S$ to $X_1,\ldots, X_r$, we have
for each $i$, $f_{j-i}\in {\mathsf{k}}_{DP}[X_1,\ldots,X_{n_i}]$: otherwise $f$ is said to have \emph{exotic} terms. See Section \ref{exoticsec} below. 
In the previous Example \ref{2to3ex} where $f = X^4Y^4 + X^3Y^3Z + Z^5$ the term $X^3Y^3Z$ is an \emph{exotic summand} in this sense. We can see this because the variable $Z$ is a partial of $f$ of order $\mathfrak{o}(Z)=5$, as we can obtain it by doing ${-x^2y^2(z-xy)\circ f=Z}$; we have $n_0=n_1=2$, so $f_{8-0}$ and $f_{8-1}$ should be written in two variables $X,Y$, thus standard terms involving $Z$ have degree at most $6$, and any term involving $Z$ of higher degree is exotic (see \cite[Definitions 3 and 4]{BJMR} or \cite[Definition 2.3, \S 2.2]{IM}). Exotic summands provide examples of unexpected behavior, as happened here with the element $z$: we have that ${z\circ f=X^3Y^3+Z^4}$, so ${z\in(0:\m^7)}$, but the term $X^3Y^3$ can be cancelled using a partial of higher order, namely ${xy\circ f=X^3Y^3+X^2Y^2Z}$. This explains why ${z-xy\in(0:\m^6)}$, a crucial fact in showing that ${z+zb}$ does not represent any non-zero class in the modules $Q(a_i)$.
\end{remark}

An example of $(A,\ell), \ell$ generic,  where $P_{A,\ell}$ does not come from the $Q(a)$-decomposition is given below (Example \ref{AA*diffex}).  There, $A$ is not strong Lefschetz, although $Q(0)$, being of codimension two, is  strong Lefschetz (Lemma \ref{cod2lem}).  Quite commonly, however, a local AG algebra  $\A$ of codimension three having $Q_\A(0)$ of codimension two is strong Lefschetz, as in the next example.
\begin{example} 
Let $R={\kk}\{x,y,z\}$, take $f = X^{[4]}Y + Z^{[3]}$, and let $\A=R/\Ann f$. Then
$\Ann f=(y^2,\, z^2-x^4y,x^5)$, the Hilbert function is $H(\A)=(1,3,3,2,2,1)$, and its symmetric decomposition is
\begin{equation*}
\mathcal{D}=\bigl(H(0)=(1,2,2,2,2,1),\, H(2)=(0,1,1,0)\bigr).
\end{equation*}
It is straightforward to check that $\ell=x+y+z$ is a SL element for both the concatenation of the $Q(a)$ decomposition, and for $A$. \end{example}

\begin{definition}\label{compatibledef}
Let $\A$ be an AG algebra of socle degree $j$ and let $\mathfrak{B}$ be a basis of $\A$ as a vector space over $\kk$. We say that $\mathfrak{B}$ is \emph{compatible with the symmetric decomposition} if, for any integers $a,i\ge0$, $Q_\A(a)_i$ is generated by the set
\[
\bigl\{v\in \mathfrak{B}\mid v\in\bigl(\mathfrak{m}^i\cap(0:\mathfrak{m}^{j+1-a-i})\bigr)
  \setminus\bigl(\mathfrak{m}^{i+1}\cap(0:\mathfrak{m}^{j+1-a-i})
  +\mathfrak{m}^i\cap(0:\mathfrak{m}^{j-a-i})\bigr)\bigr\}
\]
\end{definition}

\begin{question}\label{Jordanbasisques}\begin{enumerate}[(i).]
\item Given a pre-Jordan basis for the multiplication by an element $\ell\in\m$ in an AG algebra $\A$, compatible with the $Q(a)$ decomposition of $A$, can we adjust it by subtracting from each $b_i$ a suitable element in ${\mathsf{k}}[\ell][b_1,\ldots,b_{i-1}]$ to give a Jordan basis compatible with the $Q(a)$ decomposition of $\A^\ast$? \par
\item Let $A$ be an AG algebra and let $\ell$ be any element in $\m_A$.
Then is there a pre-Jordan basis 
$\mathfrak{B}=\{v_1,\ldots,v_n\}$ 
for the multiplication map by $\ell$ in $A$ such that  
\[
\# \mathfrak{B}\cap \m_A^i\cap(0:\m_A^{j+1-a-i}) = 
\sum_{b\ge a,k\ge i} \dim Q(b)_k
\]
for each pair $a,i\ge0$?
\end{enumerate}
\end{question}\vskip 0.2cm\noindent
{\bf Answer}  NO! to each.
Example \ref{ChrisMex} is a counterexample to (i); and Example \ref{2to3ex} is a counterexample for (ii).
\begin{example}[Representative $h\in \A$ of an element in $Q(a)$]
\label{examplereph} 
Consider $F=X^4+X^2Y$, where $\A=R/I$ with $ I=\Ann F=(xy-x^3, y^2)$, and $H(\A)=(1,2,1,1,1)$, 
\[
\D(\A)=\bigl(H(0)=(1,1,1,1,1),\, H(2)=(0,1,0)\bigr).
\] 
Here $\A^\ast={\mathsf{k}}[x,y]/(xy,y^2,x^5)$ with filtration $\A^\ast=C(0)=C(1)\supset C(2)=y\A^\ast$.
Recall that $Q(2)_1$ is a quotient of $\mathfrak{m}\cap (0:\mathfrak{m}^{4+1-2-1})=\mathfrak{m}\cap (0:\mathfrak{m}^2)$.
We have that $h=y-x^2$ satisfies $xh\circ f= y^2h \circ F=0$, so $h\in \mathfrak{m}\cap (0:\mathfrak{m}^2)$,
and represents the element $y\in Q(2)\subset\A^\ast$.  Note, however, that $y\notin (0:\mathfrak{m}^2)$ as $x^2y\circ F\not=0$, so $y\in \A$ itself is not a representative of $Q(2)_1$, but, rather, represents an element of $Q(1)_1$. \end{example}\noindent
{\bf Moral:} One cannot pass to $Q_\A(a)$ in a natural way just from $\A^\ast$; that is, the sequence $C(0)\supset C(1)\supset \ldots \supset C(j-2)$ arises from the additional structure of Gorenstein duality on the algebra $\A $. Thus, in the Example \ref{examplereph} one cannot pass first to the initial form $y$ of $y-x^2$ to see its image in $Q(2)$. Recall also that there are graded algebras $\A^\ast$ with several distinct symmetric decompositions
\cite[\S 1.6]{IM}.
\subsection{Jordan types of $\A$ and of $\A^\ast$.}\label{JTA,A*sec}
Lemma \ref{dominancelem} shows that the Jordan type for $P_{A,\ell}$ is greater or equal that of $P_{A^\ast,\ell}$ in the dominance partial order.
The examples in this section show that $P_A$ may be strictly greater than $P_{A^\ast}$, and we explore the cause. 
Our first Example \ref{symHilbex} is one in which the Hilbert function is symmetric. Then $A$ AG implies that $Q_A(0)= A^\ast$ and is a graded AG algebra (Lemma \ref{Macaulalyduallem}). We introduce further invariants of $H$ and $\D$ (Definition \ref{def:relativeLef}), and recall a key inequality $P_c\bigl(H(M)\bigr) \geq P_{\ell,M} $(Lemma \ref{contiglem}). We also explain using the Jordan strings in Example \ref{cuteex} how $A$ and $A^\ast$ may have different Jordan types.

\subsubsection{Symmetric Hilbert function, $A, A^\ast$ of different Jordan type.}\par
We will now see that even when the AG algebra $A$ has symmetric Hilbert function $H(A): \dim_{\kk}A_i=\dim_{\kk}A_{j-i}$ for $0\le i\le j/2$,  in which case $A^\ast$ is also Gorenstein, we may have $P_{A,\ell}>P_ {A^\ast, \ell}.$  This example arose from discussion with Johanna Steinmeyer.
 \begin{example}\label{symHilbex} Let $R={\kk}\{x,y,z\}$, $S={\kk}_{DP}[X,Y,Z]$ let $F=X(Y^4+Z^4)+X^3+Y^2Z^3$, let $A_F=R/\Ann F$, for which $H(A_F)=(1,3,5,5,3,1)$; the defining ideal is 
 $$I=\Ann F=(x^2-z^4,xyz,xz^2-y^2z,xy^3-yz^3,y^4-z^4),$$ with $7$  generators. We have for Jordan type of $A$ with respect to $x$, $P_{x,A_F}=(4,2^6,1^2)$, The associated graded ideal is $$I^\ast=\Ann (X(Y^4+Z^4)+Y^2Z^3)=(x^2, y^2z-xz^2, xyz ,y^4-z^4, xy^3-yz^3),$$ with JT $P_{x,A^\ast}=(2^8,1^2)$ and $\mathcal S_{x,
 A^\ast}=(2\uparrow_0^4,2\uparrow_1^3,1_{2,3}).$  We have $P_{A_F,x}>P_{A^\ast,x}$ in the dominance partial order. Note that here both $A$ and $A^\ast$ are strong Lefschetz with respect to $\ell=x+y+z$ (calculation in Macaulay2).
  \end{example}\vskip 0.2cm\par\noindent

\subsubsection{Jordan type and filtrations.}\label{filtsec}
 Recall that the Hilbert function of a finite length graded module $M$ is the sequence $H(M)=(h_0,\ldots ,h_j)$ where $h_i=\dim_{\mathsf{k}}M_i$. Given a Hilbert function sequence $H$, recall that we denote by $H^\vee$ the conjugate partition (switch rows and columns in the Ferrers graph) of the partition rearrangement of $H$  (Definition \ref{JTdef}).\par
The Example \ref{AA*diffex} in the next section below shows that the Jordan type associated to $\mathcal{D}(A)$ can be strictly greater than that associated to $A^\ast$, which was a surprise to us (Equation \eqref{inequaleq}). The explanation is that for some $a\ge 0$, the $A^\ast $-module $Q(a)$ may satisfy $P_{Q(a),\ell}<H(a)^\vee  $, rather than equality, which would be the analogue for $Q(a)$ of strong-Lefschetz. To prepare, we first give some notation and state some results about filtrations. \vskip 0.2cm\par\noindent
When $M$ is standard-graded and $H(M)$ is non-unimodular, then $H(M)^\vee$ is simply unattainable as a Jordan type. Now we will introduce the concept of contiguous partition of $H$, from \cite[Definition~2.17]{IMM1}, as the best we might do (Lemma \ref{contiglem}).
\begin{definition}[Conjugate and contiguous partitions from $H$, relative Lefschetz property]\label{def:relativeLef}\par
i. The \emph{contiguous partition} $P_c(H)$ of the Hilbert function $H$ is the partition whose parts are the lengths of the maximal contiguous row segments of the bar graph of $H$. 
\par\noindent
ii. We say a linear form $\ell\in A_1$ in a graded Artinian algebra $A$ has the \emph{Lefschetz property relative to $H$} if its Jordan type is equal to the contiguous partition of $H$:
\[
P_\ell=P_c(H).
\]
iii. The \emph{concatenation} $P_1\bigcup P_2$ of two partitions $P_1$ of $a$, and $P_2$ of $b$, is the partition of $a+b$ obtained by concatenating (placing alongside) the parts of $P_1$ and of $P_2$.
\end{definition}
Thus, for a sequence  $H_1=(1,3,5,2,3)$ we have $P_c(H_1)=(5,4,2,1,1,1)$ but for $H_2=(1,3,5,3,2)$ we have $P_c(H_2)=(5,4,3,1,1)=H_2^\vee$.  For $P_1=(4,3,1)$, $P_2=(5,3,2)$ we have $P_1\bigcup P_2=(5,4,3,3,2,1)$.\par
Given a symmetric decomposition $\D=\bigl(H(0),\, H(1),\ldots \bigr)$ of $H$, we denote by $P_\D=\bigl(H(0)^\vee, H(1)^\vee,\ldots\bigr) $ the concatenation of the conjugate partitions to $(H(0),H(1),\ldots )$. We denote by $P_c(\D)$ the concatenation of the set of $\bigl\{P_c\bigl(H(a)\bigr)\bigr\}$ of contiguous partitions. 
\begin{align}\label{Dveeeq}
P_\D&=\bigl(H(0)^\vee \cup H(1)^\vee \cup\ldots \bigr)\\
P_c(\D)&=P_c\bigl(H(0)\bigr)\cup P_c\bigl(H(1)\bigr)\cup \cdots .
\end{align}
We will need the contiguous partitions $P_c(H)$, $P_c(\D)$ whenever we deal with non-unimodal Hilbert functions $H(A)$ or non-unimodal component Hilbert functions $H(a)$.\par
We will write $P_{Q(A),\ell}$ for the concatenation of $\{P_{Q_A(a),\ell}\}$, the actual Jordan types of the graded modules $Q(a)$ with respect to $\ell$. Since $Q_A(a)$ is a graded $A^\ast$ module, we have by Lemma \ref{contiglem} below 
\begin{equation}
P_{Q(a),\ell}\le P_c(H_A(a))\le P(H_A(a)).\end{equation}

We say that a filtration $\mathcal F: M=M(0)\supset M(1)\supset \cdots \supset  M(c)\supset M(c+1)=0$ of a module $M$ over a standard graded Artin algebra $A$ is \emph{compatible} with the $\m$-adic filtration, if $\m\cdot M(i)\subset M(i)$ for each $i\in [0,c]$.
\begin{lemma}\label{comparelem} Let $A$ be an Artinian standard graded algebra, and $M$ a module over $A$ with a stratification $\mathcal F: M=M(0)\supset M(1)\supset \cdots \supset  M(c)\supset M(c+1)=0$ that is compatible with the $\m$-adic filtration. Denote the successive quotient modules by $Q_{\mathcal F}(i)=M(i)/M(i+1)$ for $i\in [0,c]$. Let $\ell\in A_1$ be a linear form. Then we have (the right side is the concatenation),
\begin{equation}\label{filtereq}
P_{M,\ell}\ge \bigcup_i \{P_{Q(i),\ell}\}.
\end{equation}
Now let $A$ be an AG algebra, and  consider the $Q(a)$ subquotients of $A^\ast$ (Definition \ref{SymmDecompdef}). Then
\begin{equation}\label{AastDeq}
P_{A^\ast,\ell}\ge \bigcup_a \{P_{Q(a),\ell}\}. 
\end{equation}
\end{lemma}
\begin{proof}[Proof of Lemma \ref{comparelem}] It suffices to show \eqref{filtereq}, and to show it for $c=1$.
Let $P=P_{M,\ell}$ have $t$ parts, and let  $M=\oplus_1^t L(i)$ be a decomposition of $M$ into $t$ $\ell$-strings (Definition \ref{JTdef}). For a string $L(i)=\bigl(m(i),\ldots ,m(i)^{p_i-1}\bigr)$ of length $p_i$ in $M$, let $n(i)=m(i)\ell^{s_i}$ be its first element in $M(1)$. Then its image in $Q(0)=M/M(1)$ is the string $m(i), \ell m(i),\ldots, \ell^{s_i-1}m(i)$ and its image in $Q(1)=M(1)$ is the string $\bigl(n(i),\ldots, n(i)^{p_i-s_i }\bigr)$. We will show 
\vskip 0.2cm\noindent
\textbf{Claim.} $Q(0)=\oplus L(i)/n(i)L(i)$ is a Jordan decomposition, and $Q(1)=\oplus n(i)L(i)$ is a Jordan decomposition of $Q(1)=M(1)$.\par\vskip 0.2cm\noindent
{\it Proof of Claim}: Suppose that $Q(0)$ has Jordan decomposition $P(0)$ and $Q(1)$ has Jordan decomposition $P(1)$ with respect to $\ell$. Then the concatenation gives a Jordan pre-basis for $M$ (we just do not have  $\ell^{p_i}m(i)=0$).  The Jordan type $P_M$ is the maximum of the partitions associated to Jordan pre-bases. so
$P_M\ge P(0)\cup P(1)$. \end{proof}\par
 It follows that there is equality in Equation \eqref{filtereq} in the case $c=1$ only if each $\ell$-string of $M(1)$ does not have a precursor $n'(i)$ in $M$ such that $\ell n'(i)=n(i)$.
\vskip 0.2cm
We recall \cite[Theorem 2.20]{IMM1},
\begin{lemma}\label{contiglem}
For a finite graded module $M$ over a standard-graded Artinian algebra $A$ with Hilbert function $H(M)$, we have for any linear form $\ell\in A_1$
\[
P_c\bigl(H(M)\bigr) \geq P_{\ell,M}.
\]
\end{lemma}

Since our subquotients $Q_A(a)$ are standard graded, Lemma \ref{contiglem} implies that   $P_{Q(a),\ell}\le P_c\bigl(H(Q(a)\bigr)$. In the Example \ref{thirdAastex} 	we have $P_{A^\ast}>P_c(\mathcal{D})$.\vskip 0.2cm\par
 We give some further examples, here and in the next section, where $P_A>P_{A^\ast}$.
In the next example $Q(1)$ is non-simple, as $H(1)=(0,1,0,1,0)$. This is reflected in the Jordan type of $A^\ast$ but not in that of $A$, as $\{z,\ell z=x^3\}$ is a length-two pre-string (Definition \ref{preJordandef}) in $A$, but $\ell z=0$ in $A^\ast$. 
\begin{example}[Cute]\label{cuteex}
Let $R={\kk}\{x,y,z\}$, $S={\kk}_{DP}\{X,Y,Z\}$, consider ${F=X^{[3]}Y^{[2]}+Y^{[3]}Z}$, and let ${A=S/I}$, with ${I=\Ann F}$. Then $H(A)=(1,3,3,4,2,1)$ and its symmetric decomposition is
\[
\mathcal{D}=\bigl(H(0)=(1,2,3,3,2,1),\ H(1)= (0,1,0,1,0)\bigr),      
\]
while $H(A)^\vee=(6,4,3,1), $
We will show that for $\ell=x+y+z$
\begin{equation}\label{cuteeq} 
H(A)^\vee>P_{\A,\ell}=(6,4,2,2)>P_{A^\ast,\ell}=(6,4,2,1,1)=P_c(\D).
\end{equation}
The generators of $I$ are given by
\[I= (xz, yz-x^3, z^2, xy^3, y^4),\] 
and the associated graded ideal $I^\ast$ defining $A^\ast$ is 
\[I^\ast=(xz,yz,z^2, xy^3,x^4,y^4),\]
the annihilator of $\langle X^{[3]}Y^{[2]},Z,Y^3\rangle$.
A calculation in Macaulay2 shows that for $\ell=x+y+z$,  or a random linear combination that the Jordan type is $P_{ A,\ell}=(6,4,2,2)$, while  $P_{A^\ast,\ell }=(6,4,2,1,1)=P_c(\mathcal{D})$. Since ${H(A)^\vee=(6,4,3,1)}$ we see that $A$ is not SL (although $Q(0)$, being of codimension two, is strong Lefschetz). We now explain the difference between $P_{A}$ and $P_{A^\ast}$.\par
First, for $A^\ast$ taking $\ell=x+y+z$, its action on generators $1,x, y^2$  gives pre-strings of lengths $(6,4,2)$ (the length $6$ pre-string is a string, since $\ell^7=0$).
In particular $\ell^2 y^2=x^2y^2\in \langle \ell^4,\ell^3x\rangle$ in $A^\ast$, so there is a length-two pre-string $\langle y^2,y^2\ell\rangle$.  We may take $z\in A^\ast$ as the generator of another pre-string of $A^*$, for which $\ell\cdot z=0$ in $A^\ast$; and the last pre-string of $A^\ast$ is $\{y^3\}$, also of length one, as $\ell y^3=0$ in $A^\ast$. Thus, 
$P_{A^\ast,\ell }=(6,4,2,1,1)$.\par
What changes for $A$ is that $\ell z=xz+yz+z^2=x^3$, giving a length two pre-string, and $\ell y^2=y^3+xy^2$ is linearly independent of  $\langle \ell^3, \ell^2x, x^3\rangle$ so $\{y^2,\ell y^2\}$ remains a length-two pre-string, yielding $P_{ A,\ell}=(6,4,2,2)$.\par
What is cute here, is that one $\ell$ pre-string of $A$ starts representing an element, class of $z$ in $Q(1)_1$, but has next bead  $\ell z=x^3$ representing a class in $Q(0)_3$, a jumping behavior.  This illustrates the principle that if $z\in Q(a)$ and $\ell z$ represents an element of a lower $Q(b), b<a$, then $\ell z$ must have a jump in degree (Remark \ref{multrem}).
\end{example}
\subsection{Several irreducible components of $\Gor(H)$.\label{2compsec}
}\par
In this section we give two initial examples of codimension three Gorenstein sequences
$H$ such that $\Gor(H)$ has several irreducible components,  Example \ref{AA*diffex} where $H=(1,3,4,4,4,2,1)$, and Example \ref{thirdAastex}, where $H=(1,3,5,4,4,2,1)$.
In between, we discuss a potential further way to distinguish components of $\Gor(H)$, using the tangent space to the Hilbert scheme (Remark \ref{Isqrem}).

In the next example the AG algebra $B$ is in an irreducible component of $\Gor(\D_2)$ that is comprised of strong Lefschetz AG algebras, whose associated graded algebras are also strong Lefschetz. The algebra $A$ is in a different irreducible component of the same $\Gor(\D_2)$, comprised of non-strong Lefschetz algebras. In the symmetric decomposition $\Gor({\D_1})$  a general enough element $C$ has the same Jordan type as $A$; because of the specialization restrictions on symmetric decompositions, no family of algebras in $\Gor(\D_2)$ can specialize to any algebra in $\Gor(\D_1)$.  Because strong Lefschetz dominates non-strong Lefschetz, there are no specializations from a family in $\Gor(\D_1)$ to an element of $\Gor(\D_2)$ that is strong Lefschetz. This contrast assures that $\Gor(H)$, $H=(1,3,4,4,4,2,1)$ has at least two irreducible components. But we show that $\Gor(\D_2)$ itself has two irreducible components, neither in the closure of $\Gor(\D_1)$, so that $\Gor(H)$ has three components. \par
  Recall from Equation \eqref{Dveeeq} that, given a symmetric decomposition $\D$, we define $P(\D)=\bigcup_{a=0}^{j-2}H(a)^\vee$, and $P_c(\D)=\bigcup_{a=0}^{j-2} P_c(H(a))$; these can be quite different from the actual concatenation $P_{Q(A)}$ of $\{P_{Q(a)}, a=0,1,\ldots,j-2\}$ with respect to a generic linear form $\ell$. Here $P_{Q(A)}$ must be dominated by $P_{A^\ast}$ since $Q(a)$ are the quotients arising from a filtration of $A^\ast$.  We now illustrate this. We also have in $A,C$ examples of a non-strong-Lefschetz AG algebra such that $Q(0)$ is strong-Lefschetz - as the codimension of $Q(0)$ is two (Lemma \ref{cod2lem}).
\begin{example}[Jordan type of $A$ and of $A^\ast$ different, $\Gor(\D_2)$ has two irreducible components, $\Gor(H)$ has three]
\label{AA*diffex}\phantom{abscked} 
 Let $R={\mathsf{k}}[x,y,z]$, $S={\kk}_{DP}[X,Y,Z]$, fix $H=(1,3,4,4,4,2,1)$ of socle degree $6$.\vskip 0.2cm\par\noindent
\textit{The algebra $C$ and $\Gor(\D_1)$:} Let 
\[
\D_1=\bigl(H(0)=(1,2,3,4,3,2,1),\, H(1)=(0,1,0,0,1,0),\, H(2)=(0,0,1,0,0)\bigr).
\]
 Let $W=X^3Y^3+ZY^4+Z^4$, then $J=\Ann W=(xz,yz-x^3,z^3-y^4)$, $C=R/J\in \Gor(\D_1)$. Here $Q_C(0)={\mathsf{k}}[x,y]/(x^4,y^4)$ and  $Q^\vee(1)=\langle Z,Y^4 \rangle, Q^\vee(2)=\langle Z^2\rangle$. Also $J^\ast=(xz,yz,z^3,y^4,x^4y,x^5)$. We have $P(\D_1)=(7,5,3,2,1^2), P_c(\D_1)=(7,5,3,1^4), $ and\par $ H^\vee=(7,5,4,3)$.  Let $\ell=x+y+z$ we have \footnote{All calculations were in {\sc{Macaulay}}, and we found $\ell=x+y+z$ is general enough in this example.} $$H^\vee>P_{C,\ell}=(7,5,3,3,1) > P_{C^\ast,\ell}=(7,5,3,2,2)>P(\D_1)>P_c(\D_1)=P_{Q(A)}.$$  
\par\noindent
{\it The algebra $A$ and $\Gor(\D_2)$:} Let 
$$\D_2=(H(0)=(1,2,3,3,3,2,1),H(1)=(0,1,1,1,1)).$$
Let $F_A=X^2Y^4+X^3YZ$, then $I=\Ann F=( z^2,xz-y^3, y^2z, x^4, x^3y^2)$, $A=R/I\in  \Gor(\D_2)$. Here  $Q_A(0)={\mathsf{k}}[x,y]/(x^3,y^5), Q^\vee(1)=\langle Z,YZ,X^3,X^3Y  \rangle$, and $A^\ast=R/(z^2,xz,y^2z,x^4,y^5,x^3y^2)$.  We have $H^\vee=P(\D_2)=P_c(\D_2)=(7,5,4,3)$. We have for $\ell=x+y+z$ (also for random $\ell$),
\begin{equation}\label{inequaleq}
P(\D_2)=P_c(\D_2)=(7,5,4,3)>P_{A,\ell}=(7,5,3,3,1)>P_{A^\ast,\ell}=(7,5,3,2,2)=P_{Q_A,\ell}.
\end{equation}
Thus $A,A^\ast$, respectively have the same generic Jordan types as, respectively, $C,C^\ast$, and they are not strong Lefschetz. 
Explanation for $P(\D_2)>P_{A^\ast}=P_{Q(A)}$: note that $P_{Q(1),\ell}=(2,2)$, there being the two $\ell$-strings on $Q(1)$ of $\langle YZ,Z\rangle$ and $\langle X^3Y,X^3\rangle $. And the component Jordan type $P_{Q(a),\ell}$ may, as here, be smaller than $P_{\D_1}$. A choice of strings for  $(A,\ell)$ giving a pre-Jordan basis is (here $z\ell^3=0$ while $z\ell^2\in \m^4$)
 $$A=\langle (1,\ell,\ldots, \ell^6); (y,y\ell,y\ell^2,y\ell^3,y\ell^4);(z,z\ell,z\ell^2),(x^2,x^2\ell,x^2\ell^2),(x^3)\rangle.$$
We will later use that the Hilbert function $H(R/I_A^2)=(1, 3, 6, 10, 12, 12,  {\color{red}{12}}, 10, 7, 3)$, so $H(I_A/I_A^2)=(0,0,2,6,8,10, {\color{red}{11}},10,7,3)$.\par
\vskip 0.2cm\noindent
 {\it The algebra $B$ and $\Gor(\D_2)$:}
Let $G=X^2Y^4+Z^5$, and set $B=R/J, J=\Ann G=(xz,yz,x^3,y^5,x^2y^4-z^5)$ satisfies $B\in\Gor(\D_2)$ but now $Q(1)=\langle Z,Z^2,Z^3,Z^4\rangle$ is cyclic, and $B, B^\ast$ are both strong Lefschetz. That is, for $\ell=x+y+z$
$$ H^\vee=(5,4,3,3)=P_{B,\ell}=P_{B^\ast,\ell}.$$
We will need that the Hilbert function of $R/I_B^2$ is $(1, 3, 6, 10, 12, 12, {\color{red}{12}}, 8, 6, 4, 2)$, so $H(I_B/I_B^2)=(0,0,2,6,8,10, {\color{red}{11}},8,6,4,2)$.\par\vskip 0.2cm\noindent
{\it Dimension of $\Gor(\D_1)$}: Choose $\langle X,Y\rangle\subset S_1$ a $\mathbb P^2$ choice.  Let  $F$ be a dual generator of $C\in \Gor(\D_1)$. By Theorem \ref{ZTlem}, Equation \eqref{dimZteq} the dimension of the (not necessarily graded) quotients $Z_{H(0)}=n_0-4=12$. By \cite[Lemma 1.40]{IM}, $H(1)=(0,1,0,0,1)$ now implies that the $Z$ portion of $F_5$ is linear in $Z$. So it has the form $ZW, W\in {\kk}[X,Y]_4$ giving a dimension 5 fibering. We have the space $zR_1$ free to act on $F_4$, and
we must have $\dim (zR_1)\circ F_4\mod {\kk}[X,Y]_2=1$, this requires $F_4=L^4\mod {\kk}[X,Y]_4 $ (as we have already accounted for the $ {\kk}[X,Y]_4 $ portion of $F_4$, thus we choose $L=az+bx+cy$, a three-dimension choice. We may add arbitrary elements from $S_3,S_2$ not already in $R\circ (F_6+F_5+F_4)$ without changing $H$, so we may add $(r_2-4)+(r_3-4)=(6-4)+(10-4)$, obtaining a total dimension of $2+12+5+3+8=30$ for $\Gor(\D_1)$.\vskip 0.2cm\noindent
\vskip 0.2cm\noindent
{\it Dimension of $A$ component of $\Gor(\D_2)$.} Choosing $\langle X,Y\rangle$. Let $F$ be  a dual generator of an algebra with the Jordan type of $A$ in $\Gor(\D_2)$: then by Equation \eqref{dimZteq}, again the dimension of the not necessarily graded quotients $Z_{H(0)}=n_0-3=12$. As for $\D_1$ we are choosing a $ZW, W\in {\kk}[X,Y]_4 $, general enough, so this is a 5 dimensional choice (we care about the proportionality coefficient). The degree-4 component of $F$ \label{I2argument} in $Z$ must be annihilated by $(I_2)^2= \langle z^2, z\ell)^2\rangle$, a 3-dimensional space and is viewed mod $R\circ F_{\ge 5}$ of dimension $4$ so gives a $10-4-3=3$ dimensional space.  Finally we may add lower degree terms in $Z$: giving $(r_2-4)+(r_3-4)=8$; we  have for dimension of this component  $2+12+5+3+8=30$, again.\vskip 0.2cm \par\noindent
{\it Dimension of $B$ component of $\Gor(\D_2)$.} The choice of  $F_6$, as before is fibred over $\mathbb P^2$ by a 12-dimensional $Z_{H(0)}$. Now choose the $Z$-related portion of $F_5$ as $L^5$, a 3 dimensional choice. We may add on elements of 
$(I^2)^\perp$ to $F_4$, as we must have $I_2\circ F_4\subset (\Ann F)_2$. Here $(I^2)_4=(xz,yz)^2$ (since, after coordinate change $F_6+F_5$ is a connected sum), so, again, as for $A$, the $Z$-portion of $F_4$ is a
$10-4-3= 3$-dimensional choice. Finally $F_3+F_2$ adds 8 to the dimension, giving  $2+12+3+3+8=28$ dimension family.\par
We have now shown
\begin{claim}\label{Dtwocompclaim} 
\textit{$ \Gor(\D_2)$ itself has two irreducible components}, one corresponding to $F_5=ZG_4\mod {{\kk}}[X,Y]$ as for the algebra $A$;  the other component satisfies $F_5=L^5 \mod {{\kk}}[X,Y]$, as for the algebra $B$. The components have respective dimensions 30 for $A$ and 28 for $B$.
\end{claim}
\begin{proof}[Proof of Claim] An algebra in the family related to $A$ cannot specialize to a general element related to $B$ by the closure properties of Jordan type (Lemma \ref{dominancelem}). But the $B$-component has smaller dimension, 
so a family in the $B$-related component cannot specialize to a general element in the $A$-component.
\end{proof}\par\noindent
\textit{Specialization: proof that $\Gor(H)$ has three components.}\par
Since $H_C(0)_3=4>H_B(0)=3$ we have by Corollary \ref{D0cor} that no subfamily of $\Gor(\D_2)$ can specialize to an algebra in $\Gor(\D_1)$; this and that $P_B=(7,5,4,3)$ dominates $P_C=(7,5,3,3,1)$ shows that we have at least two irreducible components of $\Gor(H), H=(1,3,4,4,4,2,1)$, corresponding to these two symmetric decompositions; but we have just shown that $\Gor(\D_2)$ itself has two irreducible components, neither in the closure of $\Gor(\D_1)$.
\end{example}
\subsubsection{Tangent space to the Hilbert scheme $\Hilb^n(\mathbb A^2)$ at $A$.}
We introduce this mainly as a theme to further explore elsewhere: we do not use this tangent space in an essential way in this paper.
\begin{remark}\label{Isqrem}
\textit{Negative deformations, tangent space.} The tangent space to the Hilbert scheme $\Hilb^n(\mathbb{A}^2)$ at a point corresponding to the length-$n$ Artinian algebra $A=R/I$ is $\Hom(I,A)$; for a graded AG algebra we would have $\Hom(I,A)\cong I/I^2$ with a reversal in degree \cite[Theorem 3.9, Remark 3.10]{IK}; however we are not sure how this might extend to the non-graded case.\par 
Calculation in Macaulay for the algebras $A$, $B$, $C$ above, all of Hilbert function $H(A)=(1,3,4,4,4,2,1)$, shows, remarkably, an equality in the lengths
$|H(R/I_A^{\,2})|=|H(R/I_B^2{\,2})|=|H(R/I_C^{\,2})|$, and in fact, we have 
$|H(I/I^2)=3\cdot |H(R/I)|$ for each of the ideals $I=I_A,I_B,I_C$.\footnote{A result of J.O. Kleppe shows that as a consequence of the D. Buchsbaum-D.Eisenbud Pfaffian structure theorem applied to graded AG algebras of codimension three, the Hilbert function $H(R/I^2)$ depends only on $H(R/I)$,  (\cite[Proposition 2.5]{Kl1}, \cite[Proposition 4.25]{IK}). This Hilbert function result does not generalize to nongraded AG algebras of codimension three; however in each example we have checked the length $|H(I/I^2)|=3|H(R/I)|$ (See Note \ref{discussnote} p. \pageref{discussnote}).}\par
\end{remark}
That the dimension $\dim_{\kk}(I/I^2)=3\dim(R/I)$, is a consequence of the smoothability of all codimension three Artinian Gorenstein algebras.
We note:\vskip 0.2cm
\par\noindent
{\bf Principle.}
For a constant length family $ \Xi(t), t\in T$, of Artinian algebras $\Xi(t)=R/I(t)$, the sum 
$\sum_{i\ge i_0}H(\Xi(t))_i$ for a fixed $i_0$ is an upper semicontinuous function. Higher values are more special: one cannot deform from a family with higher sum to one of lower sum. The reason is, assuming that the socle degrees of $\Xi(t)$ are $j$ or below, so working in the finite dimensional quotient ring $R^\prime= R/\m^{j+1}$ that for $\Xi(t)=R^\prime/I(t)$ the dimension of the intersection  $I(t)\cap \m^{i_0}\subset R^\prime$ is semicontinuous.
\par
 This implies
\begin{lemma}\label{tangentlem} 
Consider two AG algebras $A$, $B$ of codimension three, and assume $H(A)=H(B)$. Let $\AAA=\{A(t), t\in T\}$,
 be a family of algebras satisfying $H\bigl(A(t)\bigr)=H(A)$ and $H\bigl(I_{A(t)}/I_{A(t)}^{\,2}\bigr)=H(I_A/I_A^{\,2})$. Assume that there is an integer $i_0$ such that
\begin{equation}
\sum_{i\ge i_0}\dim_{\kk}\bigl((I_{A(t)})_i/(I_{A(t)})_i^{\,2}\bigr)>\sum_{i\ge i_0}\dim_{\kk}\bigl((I_B)_i/(I_B)_i^{\,2}\bigr).
\end{equation}
Then there can be no specialization from the family $\AAA$ to $B$. 
\end{lemma}
\begin{proof} We have 
\begin{equation}
H(R/I_A^{\,2})=H(I_A/I_A^{\,2})+H(R/I_A)=H(I_A/I_A^{\,2})+H(A),
\end{equation}
so, given $H(A)=H(B)$ and the fact $\dim_{\kk}(I_A/I_A^{\,2})=3\dim(R/I_A)$ for height three local AG algebras, we conclude the result from the Principle above.
\end{proof}\par 
For, example, in codimension two or less, the curvilinear length $n$ local algebras of Hilbert function $H=(1,1,\ldots, 1_{n-1})$ specialize to all other local algebras of length $n$, but no family of non-curvilinears can specialize to a curvilinear algebra.\footnote{Reason: each colength $n$ ideal $I$ of $R$ contains $\m^n$, so we may work with $I/\m^n\subset R/\m^n$. The codimension of $I\cap \m^{i_0}/\m^n$ in $\m^{i_0}/\m^n$ for  a fixed $i_0\le n$ is $\sum_{i\ge i_0}H(R/I)$. So a larger sum greater than $c$ implies a smaller intersection of $I$ with a fixed ideal $\m^{i_0}$, which is an open condition on ideals $I/\m^n$ of fixed length in $R/\m^n$.}

We have from calculation in \textsc{Macaulay2} for the algebras $A$, $B$, $C$ above in Example \ref{AA*diffex}
\vskip 0.3cm\noindent
$\begin{array}{c||c|c|}
R/I&H(R/I^2)&H(I/I^2)\\\hline
A&(1, 3, 6, 10, 12, 12,  {\color{red}{12}}, 10, 7, 3)&(0,0,2,6,8,10,{\color{red}{11}},10,7,3)\\\hline
B&(1, 3, 6, 10, 12, 12, {\color{red}{12}}, 8, 6, 4, 2)&(0,0,2,6,8,10,{\color{red}{11}},8,6,4,2)\\\hline
C&(1, 3, 6, 10, 12, 12, {\color{red}{11}}, 9, 7, 4, 1)&(0,0,2,6,8,10,{\color{red}{10}},9,7,4,1)\\\hline
\end{array}$
\vskip 0.2cm
Here, we may compare $H(R/I_A^{\,2}) $ and $H(R/I_B^{\,2})$ as within $\Gor(H)$ the lengths of $H(A)$ and of $H(R/I_A^{\,2})$ are constant. Thus, it follows from the table above that there are no specializations from a family like $B$ to $A$ ($i_0=8$) nor $C$ to $B$ ($i_0=7$) nor  $B$ to $C$ ($i_0=10$).  Warning: we have not checked the $H(R/I^{\,2})$ for generic algebras of these irreducible components of $\Gor(H)$. 
 
\subsubsection{When $H(A)^\vee>P_A>P_{A^\ast}>P_c(\D)$, and a second $\Gor(H)$ with two irreducible components.}

The following example was discussed in \cite{IM} with reference to the decomposition $\D_1$, but now we consider the Jordan type. We will see that $A$ is not strong Lefschetz, and that for a generic linear form $\ell$ (see Definition \ref{def:relativeLef}), we have
 $H(A)^\vee>P_{A,\ell}>P_{A^\ast,\ell}> P_c(\mathcal{D}_1).$ 
We also show that the family $\Gor(H)$ has two irreducible components,
corresponding to the two symmetric decompositions $\D_1$ and $\D_2$ of $H$. This is  the main theme we will explore in Section 3.
\begin{example}\label{thirdAastex} \par
(i). We recall from  \cite[Example 1.31]{IM}, let $R={\kk}\{x,y,z\}, S={\kk}_{DP}[X,Y,Z]$ and $F=X^{[3]}\cdot Y^{[3]}+Z(X^{[4]}+Y^{[4]})\in S$. Then $A=R/\Ann F$ has Hilbert function $H_F=(1,3,5,4,4,2,1)$ and symmetric  decomposition
\begin{equation}\label{geneq1}
\D_1=\big(H_F(0)=(1,2,3,4,3,2,1),\ H_F(1)= (0,1,0,0,1,0),\ H_F(2)=(0,0,2,0,0)\big).     
\end{equation}
We will show that, while $H(A)^\vee=P_c(H(A))=(7,5,4,3,1)$, for a generic $\ell$, 
\begin{equation}\label{alldifferenteq}
P_{A,\ell}=(7,5,3,3,1,1)> P_{A^\ast,\ell}=(7,5,3,2,2,1)=\mathcal{D}^\vee >P_c(\D_1)=(7,5,3,2,1,1,1).
\end{equation}
 The ideal $I=\Ann F = (z^2, xyz, x^2z-xy^3, y^2z-x^3y, x^4-y^4)$; the associated graded ideal $I^\ast$ defining $A^\ast$ is $I^\ast=(z^2, xyz, x^2z, y^2z, x^4-y^4,xy^4,yx^4).$

Here
$
Q_F^\vee(1)=Q_F(1)\circ F=\langle x^4\circ f,z\circ f\rangle=\langle
Z,X^{[4]}+Y^{[4]}\rangle,\quad
Q_F^\vee(2)=Q_F(2)\circ F= \langle (-zx+y^3)\circ f, (-zy+x^3)\circ
F\rangle=\langle XZ,YZ\rangle .$
\par 
Note, $H(A)^\vee=(7,5,4,3,1)$, $\mathcal{D}_1^\vee=(7,5,3,2,2,1)$ while $P_c(\D_1)=(7,5,3,2,1,1,1)$. Also, for any $\ell, P_{Q_F(1),\ell}=(1,1)=P_c(H(1)$ so each $P_{Q_F(a),\ell}$ achieves its maximum $P_c(H(a))$ under Lemma \ref{contiglem}. Nevertheless, as we shall see, $H(A^\ast)>P_c(\mathcal{D_1})$.\par
The structure of $Q_F(1)$ suggests that the Jordan type of $A^\ast$ might be quite special compared to that of $A$. This is confirmed by the remark, that for any non-zero linear form $\ell=ax+by+cz$, the homomorphism
$m_{\ell^2}: A^\ast_1 \to A^\ast_3$ has a kernel, as $z\ell^2=0 \mod I^\ast$. We may use as $\mathsf{k}$-basis
\[
A=\langle 1;x,y,z;x^2,xy,y^2,zx,zy;x^3,x^2y,xy^2,y^3;x^4,x^3y,xy^3,y^4;x^3y^2,x^2y^3;x^3,y^3\rangle.
\] 
For $\ell=ax+by+cz$ general enough both $A,A^\ast$, have linearly independent strings $L_1=\{1,\ell,\ldots, \ell^6\}$, $L_2=\{x,x\ell,\ldots, x\ell^4\}$. A pre-Jordan basis for $A$ is
pre-Jordan basis  $$L_1,L_2,L_3=\{z,z\ell,z\ell^2\},L_4=\{y^2,y^2\ell,y^2\ell^2\}, L_5=\{x^2\},L_6=\{x^4\}=\{y^4\}.$$
To establish this, note that $z\ell^3\equiv z(x^3+y^3)=x^3y^2+x^2y^3\}$ is in $\langle L_1\oplus L_2\rangle$ as $H(A)_5=~2$. Here $x^2\ell$ does not kick off a class $\langle x^4\rangle$, we know that $A_4$ is already full. This shows $P_{A,\ell}=(7,5,3,3,1,1)$.
For $A^\ast$ we have strings $L_1,L_2,L^\prime_3=\{z,z\ell\}$ (as $z\ell^2\in I^\ast$). Also $L^\prime_4=\{y^2,y^2\ell,y^2\ell^2\} $ (here $A_5$ is already in the span of $L_1,L_2$). We have 
$L^\prime_5=\{zy,zy\ell\},L^\prime_6=\{xy,xy\ell\},L^\prime_7=x^2$, giving $P_{\ell,A^\ast}=(7,5,3,2,2,1)$.\par
Calculation in Macaulay 2 confirms that $P_{A,\ell}=(7,5,3,3,1,1)$ for $\ell$ random or $\ell=(x+y+z)$;
$P_{A^\ast,\ell}=(7,5,3,2,2,1)=\mathcal{D}_1^\vee >P_c(\mathcal{D}_1) $ for $\ell $ random and $P_{A^\ast,\ell}=(7,5,3,2,1,1,1) =P_c(\mathcal{D}_1)$ for  $\ell=(x+y+z)$. \vskip 0.2cm \par
(ii). The algebra $B=R/\Ann G, G=X^{[2]}\cdot Y^{[4]}+Z^{[5]}+(Y+Z)^{[3]})\in S$ has the same Hilbert function $H=(1,3,,5,4,4,2,1)$ and decomposition 
$$\D_2=\D_B=\left(H(0)=(1,2,3,3,3,2,1),H(1)=(0,1,1,1,1), H(2)=(0,0,1)\right).$$
Since $\D_2< \D_1$ from $A$ in the order of Lemma \ref{trianglelem} there can be no specialization from a subfamily of $\Gor(\D_2)$ to an element of $\Gor(\D_1).$  But it is easy to check that $P_B=(7,5,4,3,1)$, strictly greater in the dominance partial order than $P_A=(7,5,3,3,1,1)$ (Lemma \ref{dominancelem}) so there is no specialization from a family in $\Gor(\D_1)$ to a general element of $\Gor(\D_2)$. It is not hard to show that each family is irreducible. So $\Gor(H)$ has two irreducible components, each corresponding to a symmetric decomposition of $H$.
\end{example}
\par\noindent
Examples \ref{symHilbex},~\ref{cuteex},~\ref{AA*diffex} and \ref{thirdAastex} are  of algebras with $P_A> P_{A^\ast}$.
Example~\ref{AA*diffex} is of a non-SL algebra $A$ such that $Q_A(0)$ is SL.
Johanna Steinmeyer asked the following in our MFO (Oberwolfach) discussion of October 2020.
\begin{question}\label{Steinques}(a) Can we find a strong Lefschetz AG algebra $A$ such that $A^\ast $ is not SL? (b) Can we show that $A$ SL implies $Q_A(0)$ is SL?
\end{question}
Both questions are open. 

\subsection{Exotic summands.}\label{exoticsec}\label{exoticsec}
When parametrizing families of AG algebras with given Hilbert function $H$, we can use both the dual generator $f$ and the ideal $\Ann f$. Since we can compute the Hilbert function and its symmetric decomposition $\D$ from the vector space $R\circ f$, of partials of $f$, it is crucial to understand how changes in the summands of $f$ in each degree may affect $H$ and $\D$. In a preprint that led to \cite{BJMR}, A. Bernardi and K. Ranestad noticed that some summands had variables that were unexpected in high degrees, given the decomposition $\D$. They termed them \emph{exotic summands}.  
For convenience we repeat the definition, given in Remark \ref{exoticbrem}.
\begin{definition}[Exotic summand of $f$]\label{exoticdef}(See \cite[Definition~3 and 4]{BJMR}, \cite[Definition 2.3, \S 2.2]{IM}). 
 Let $H$ be a Gorenstein sequence of socle degree $j$ and $\D=\bigl(H(0),\, H(1),\ldots,\, H(j-2)\bigr)$ a symmetric decomposition of $H$.  Let $n_i=n_i(\D)=\sum_{u=0}^i H(u)_1$. 
We say that $f\in \mathcal S$ determining $A=R/\Ann f\in \Gor(\D)$ admits a \emph{standard form} if, after a change of variables in $\mathcal S$ to $X_1,\ldots,X_r$,
for each $i$, $f_{j-i}\in {\mathsf{k}}_{DP}[X_1,\ldots,X_{n_i}]$. Otherwise $f$ is said to have \emph{exotic} terms.\footnote{For $f=(X-Y)^7+Y^5$, a change of variables $X_1=X-Y$ leads to $f=X_1^7+Y^5$, which is in standard form. When there is an exotic summand, no linear change of variables leads to $f$ in standard form.}
\end{definition}
For examples, see Remark \ref{exoticbrem} earlier and Example \ref{ex:exoticsummand} below.
 It is known that up to isomorphism of local algebras we may choose the dual generator to be in standard form: more precisely, given the AG algebra  $A$ we may determine $B=R/J$, $J=\Ann G$ such that $B\cong A$ and $G$ is in standard form. This is shown in \cite[Theorem~5.3]{I2}, but with more detail in \cite[Theorem~2.7]{IM}).\footnote{This standard form is not obtained by simply omitting the exotic summands, which would not in general preserve the isomorphism class.}

\begin{example}
\label{ex:exoticsummand}
Let $R={\kk}\{x,y,z\}$, $S=\kk_{DP}[X,Y,Z]$, and consider ${A=R/\Ann f}$, with ${f=X^3Y^3+YZ^3}$ in $S$. We can see that $A$ has Hilbert function ${H=(1,3,5,5,3,2,1)}$, with symmetric decomposition
\[
\D=\bigl(H(0)=(1,2,3,4,3,2,1),\, H(2)=(0,1,2,1,0)\bigr).
\]
If we take a close look at $R\circ f$, we see that the action of monomials $x^py^q$, with ${0\le p,q\le 3}$ on $f$ explains $H(0)$, while the partials 
\begin{align*}
z\circ f&=YZ^2 & yz\circ f&=Z^2 & yz^2\circ f&=Z\\
& & z^2\circ f&=YZ
\end{align*}
account for $H(2)$. Note that $x\circ f$ and $y\circ f$ are partials of degree $5$, while $z\circ f$ is a partial of degree $3$, which helps to explain why $z\circ f$ and its subsequent partials contribute to $H(2)$. This is not surprising, since $Z$ only occurs in $f$ in a summand of degree $4$. We say that $Z$ is a partial of order $3$, in the sense that it is obtained by the action of $yz^2$, an element of order $3$, and cannot be obtained by the action of an element of higher order. (For instance, even if ${z^3\circ f}=Y$, we know that $Y$ is a partial of order $5$, because ${x^3y^2\circ f=Y}$.) 

Now it would be tempting to think that any dual generator $g$ defining an AG algebra ${B=R/\Ann g}$ with the same Hilbert function $H$ and decomposition $\D$ would have summands in degrees $6$ and $5$ that could be written with only two variables (up to a linear change of variables, if necessary). This is true for degree $6$, but not for degree $5$. For example, if we take ${g=X^3Y^3+X^2Y^2Z}$ then $B$ has Hilbert function $H$ with decomposition $\D$. Although $Z$ occurs in degree $5$, it is a partial of order $3$, since ${(-xyz+x^2y^2)\circ g=Z}$, and we cannot obtain it with the action of an element of higher order. The reason for this behaviour is that we can take ${\varphi=xy}$ and write ${g=g_{\mathrm{st}}+g_{\mathrm{ex}}}$, with
\begin{align*}
g_{\mathrm{st}} & = X^3Y^3-XYZ^2-Z^3\\
g_{\mathrm{ex}} & = Z(\varphi\circ X^3Y^3) + Z^2(\varphi^2\circ X^3Y^3) 
  + Z^3(\varphi^3\circ X^3Y^3)
\end{align*}
Note that ${(z-xy)\circ g_{\mathrm{ex}}=0}$. We say that $f$ and $g_{\mathrm{st}}$ are in the standard form, and we call $X^2Y^2Z$ an exotic summand of $g$ (see \cite[Definitions 3 and 4]{BJMR} or \cite[Definition 2.3, \S 2.2]{IM}). 
\end{example}

The following result is a consequence of \cite[Proposition 6]{BJMR}, which implies that an element $f$ defining an algebra $A=R/\Ann f\in\Gor(\D)$ can be written $f=f_{st}+f_{ex}$ where $f_{st}$ is standard, and $f_{ex}$ can be expressed using $f_{st}$. For any ${a\ge0}$, let ${n_a=\sum_{i\le a}H(a)_1}$.\vskip 0.2cm

\begin{lemma}[Exotic count, two to three variables]
\label{exoticlem}
Let $H$ be a Gorenstein sequence with $H_1=3$; let $a$ be an integer at 
least $2$, and let $\D$ be a symmetric decomposition of $H$ satisfying 
$n_{a-1}=2$ and $n_a=3$. 
Let $R/\Ann f\in \Gor(\D)$. Then we can choose a basis $X,Y,Z$ for 
$S_1$ such that $X$ and $Y$ are partials of $f$ of order at least $j-a$ 
and $Z$ is a partial of order $j-a-1$ -- in other words, 
${X,Y\in\m^{j-a}\circ f}$ and ${Z\in(\m^{j-a-1}\circ 
f)\setminus(\m^{j-a}\circ f)}$. Furthermore, we may write 
$f=f_\mathrm{st}+f_\mathrm{ex}$, where $f_\mathrm{st}$ is standard 
(i.e., for any degree ${d>j-a}$, 
${(f_\mathrm{st})_d\in\kk_\mathrm{DP}[X,Y]}$) and
\begin{equation}\label{exotic1eq}
f_\mathrm{ex}=Z\varphi(f_\mathrm{st})+Z^{[2]}\varphi^2(f_\mathrm{st})+Z^{[3]}\varphi^3(f_\mathrm{st})+\cdots
\end{equation}
where $\varphi\in {\mathsf{k}}\{x,y\}_{\ge2}$ and the degree of 
$\varphi(f_\mathrm{st})$ is at least ${j-a}$.\par
In particular, if we fix $f_\mathrm{st}$, and let 
$A(f_\mathrm{st})\subset \Gor(\D)$ be
\[
A(f_\mathrm{st})=\{A=R/\Ann f\mid f=f_0+f_{ex}\},
\]
then $A(f_\mathrm{st})$ is an affine space -- parametrizing 
$\varphi(f_\mathrm{st})$ in Equation \eqref{exotic1eq}; it has dimension
\begin{equation}\label{exotic2eq}
\dim A(f_\mathrm{st})=\sum_{b=0}^{a-1}\sum_{u=j-a}^{j-b-2} H(b)_u.
\end{equation}
\end{lemma}
\begin{proof} The Equation \eqref{exotic1eq} results from 
\cite[Proposition 6]{BJMR}, where the number of variables is ${n=3}$, 
and all variables occur as partials od $f$, since ${n_a=3}$. The algebra 
$A\in A(f_\mathrm{st})$ is determined uniquely by the choice of terms 
$\varphi(f_\mathrm{st})$ of $f_{ex}$ in Equation \eqref{exotic2eq}. 
These are parametrized by their top degree components.
\end{proof}\par

\begin{remark}\label{exoticrem}
Note that if $n_{0}=2$, we have ${H(b)_{j-b-1}=0}$ for any ${b>0}$ so 
the sum in \eqref{exotic2eq} becomes ${\sum_{u=j-a}^{j-2} H_u}$.
\end{remark}

We apply Lemma \ref{exoticlem} and the Remark in Theorem \ref{2.1thm}, as part of determining the dimension of the irreducible components of $\Gor(H), H=(1,3,4,4,3,2,1)$, and similarly in Theorem~\ref{3compthm} where $H=(1,3,4^k,3,2,1), k\ge 1$. \par
\subsection{Relatively compressed $a$-modifications}\label{rcsec}
We study the effect of adding generically chosen terms of degree $j-a$ to a dual generator $F$.
We suppose that a dual generator $F=F_j+\cdots +F_{j+1-a} \in S={\kk}_{DP}[X_1,\ldots, X_r]$ has been fixed for an AG algebra $A=R/I, R={\kk}\{x_1,\ldots,x_r\}$. Let $F'=F+G, G\in S_{\le j-a}$; then we 
term $A'=R/\Ann F' $ an $a$-modification of $A$. 
Given $F$, we define $M(A,a)$ to be the following sequence, symmetric about $(j-a)/2$. Recall $r_i=\dim_{\kk}R_i$, and we denote by $H=(h_0,\ldots, h_j$ the Hilbert function $H=H(A)$. 
\begin{equation}\label{Maeq}
M(F,a)=\begin{cases} r_i-h_i &\text { for } i\le (j-a)/2\\
M(F,a)_{j-a-i}& \text { for } i> (j-a)/2.
\end{cases}
\end{equation}
In words, we symmetrize the difference $(H(R)-H(A))_{\le {j-a}/2}$.  For example if $R={\kk}\{x,y,z\}$ and
$j=7, a=1$ and $H(A)=(1,2,2,2,2,2,1)$, then $M(F,1)=(0,1,4,4,1)$.  
\begin{definition}\label{rcdef}We say that $F'= F+G, G\in S_{j-a}$ is a relatively compressed $a$-modification of $F$ if $H(A')=H(A)+M(A,a)$. 
\end{definition}
The next Lemma is a result of J. Emsalem and the first author \cite[Theorem 3.3]{I2}, see also \cite[\S 1.2, and Proposition 1.18]{IM}.  
\begin{lemma}\label{maxlem} Let $F' $ satisfy $F'= F+G, F=F_j+\cdots +F_{j+1-a}$ be given and let $G'\in S_{\le j-a}$, let $A'=R/\Ann F' $. Then 
\begin{equation}\label{maxeq}
H(A' )\le H(A)+M(A,a).
\end{equation}
 Also, if $G_{j-a}\in S_{j-a}$ is general enough we have equality in Equation \eqref{maxeq}. The family of relatively compressed $a$-modifications of $F$ has dimension $\sum_{i=0}^{j-a} (r_i-H(A)_i)$, and is parametrized by general enough choices of elements $G_{\le j-a}\in S_{\le j-a}\mod \m^a\circ F$.
 \end{lemma} 
 \begin{example}\label{maxex} Let $R={\kk}\{x,y,z\}$ and $A=R/I, I=(z, xy, x^6-y^6)=\Ann (X^5+Y^5)$ of Hilbert function $H=(1,2,2,2,2,2,1)$. Take $G_5=\sum_1^4 L_i^{[5]}$ where $L_1,\ldots, L_4$ are four general enough linear forms in $R_1$. Then $A' = R/\Ann (F+G)$ will have Hilbert function $H'=H(A' )=H+(0,1,4,4,1)=(1,3,6,6,3,2,1)$.  The elements  $F' =F+G$ with $F_6=X^6+Y^6$ and $G\in S_{\le 5}$ are parametrized by an open dense in an affine space of dimension $$(1,3,6,10, 15,21)-(1,3,6,6,3,2)=((10_3-6_3)+(15_4-3_4)+(21_5-2_5)=35.$$
 The AG algebras of HF $H'$ with decomposition $\mathcal D=(H(0)=H, H(1)=(0,1,4,4,1)$ are parametrized by
 a variety of dimension four (choose $\langle X,Y\rangle$ from $R_1$, then pick $X,Y$) for the graded AG algebra quotient of $R$ having HF  $H$. So the AG algebras of decomposition $\mathcal D$ form an irreducible family of dimension $39$. This is the unique symmetric decomposition for $H' $, so $\Gor(H' )$ is irreducible of dimension 39.
 \end{example}
 \begin{question} Given the Jordan type $P_{A,\ell}$ what are the possible Jordan types
 $P_{A^\prime,\ell}$ of a relatively compressed $a$-modification $A^\prime$ of $A$?
 Does it satisfy an inequality involving the JT of $A$ and the difference $H(A^\prime)-H(A)$ where $A^\prime$ is the RCM of $A$?
 \end{question}
We use relatively compressed modifications in the proof of (iv) of Theorem 3.1, p. 29.

\section{Applying Jordan type and symmetric decomposition for a local AG algebra.}\label{applysec}
 The main results of the paper introduce a new method of showing that $\Gor(H)$ can have several irreducible components, corresponding to the symmetric decompositions of $H$: namely the combination of the dominance  partial order on Jordan types (Lemma \ref{dominancelem}) and the partial order on symmetric decompositions (Lemma \ref{trianglelem} and Corollary \ref{D0cor}). We have already used these to show the existence of $\Gor(H)$ having two irreducible components, in Section \ref{2compsec},  Examples \ref{AA*diffex} and \ref{thirdAastex}.  The following examples have at least three irreducible components. A simplifying aspect of Theorem \ref{2.1thm} is that each component  $H(Q(a))$ is unimodal, leading to equality between the Jordan types of the algebras and their associated graded algebras, in contrast to Example \ref{AA*diffex} above.
 \subsection{Examples of AG families $\Gor(H)$ having three irreducible components.}\label{3compsec}
  We give next our first example with three irreducible components, which we describe in detail. We extend the result in Theorem \ref{3compthm} to $H=(1,3,4^k,3,2,1), k\ge 1$, giving an infinite sequence of such examples. In the following Theorem the proof that $\Gor(\D_1),\Gor(\D_2)$ and $\Gor{\D_3}$ must determine at least three irreducible components of 
  $\Gor(H)$ is immediate from Lemmas \ref{dominancelem} and \ref{trianglelem} once we determine the generic Jordan type for each $\Gor(\D_i)$; but we go further to describe the structure of each $\Gor(\D_i)$, showing it is irreducible, and specifying their dimensions.
\begin{theorem}[$Gor(H)$ with three irreducible components, different Jordan types, one SL]\label{2.1thm}
The Gorenstein sequence $H=(1,3,4,4,3,2,1)$ has  three symmetric decompositions (a),(b),(c).
\begin{enumerate}[(i.)] 
\item (a.) A generic algebra in $\Gor(\D_1)$, 
\[
\D_1=\bigl(H(0)=(1,2,3,4,3,2,1),\, H(1)=H(2)=0,\, H(3)=(0,1,1,0)\bigr)
\]
will have Jordan type $P_1= (7,5,3,2,1)$.  Take $F=X^3Y^3+Z^3$.
  Then $A=R/\Ann F$, $\Ann F=(xz,yz,z^3-x^3y^3,x^4,y^4)$ is in $ \Gor(\D_1)$.\par 
  Also having this decomposition is $A^\prime=R/\Ann F^\prime$, $F^\prime=X^6+Y^6+(X+Y)^{[6]}+(X-Y)^{[6]}+Z^3$.\par
  Here $H^\vee=(7,5,4,2)$, $A^\ast=R/\Ann(X^3Y^3,Z^2)$, and for $\ell=x+y+z$, we have $P_{A,\ell}=P_{A^\ast,\ell}=P_{\D_1}=P_1$.\vskip 0.2cm\par
  (b.) A generic algebra in  $\Gor(\D_2)$
\[
\D_2=\bigl(H(0)=(1,2,3,3,3,2,1),\, H(1)=0,\, H(2)=(0,1,1,1,0)\bigr)
\]
will have Jordan type $P_2= (7,5,3,3)$. Let $G=X^6+Y^6+(X+Y)^{[6]}+Z^4$, 
 then $\mathcal B=R/\Ann G$, $ \Ann G=(zx,zy,xy(x-y),2z^4-x^6,x^5+y^5-3xy^4)$ is in $\Gor(\D_2)$. We have $\mathcal B^\ast=R/\Ann (G-Z^4,Z^3)$ and $P_{\mathcal B,\ell}=P_{\mathcal B^\ast,\ell}=P_{D_2}=P_2$. \par
 (c.) Finally, a generic algebra in $\Gor(\D_3)$,
\[
\D_3=\bigl(H(0)=(1,2,2,2,2,2,1),\, H(1)=(0,1,2,2,1,0)\bigr),
\]
will have strong Lefschetz Jordan type $P_3= (7,5,4,2)$.
Take $\C=R/\Ann W$, $W=X^6+Y^6+(X+Y)^{[5]}+Z^5$. $\Ann W=\bigl(xz,yz,xy(x-y),z^5-x^6,x^6-y^6\bigr)$. Then $\C\in \Gor(\D_3)$. Here $\C^\ast=R/\Ann(W-Z^5,Z^4)$ and
with $\ell=x+y+z$  $H^\vee=P_{\C,\ell}=P_{C^\ast,\ell}=P_{\D_3}=P_3$.\par
There are no further symmetric decompositions of $H$.
\item  In the partial order of Equation \eqref{PODeq} we have $\D_1\ge \D_2\ge \D_3$ (see Figure \ref{symdecfig}), and 
\begin{equation}\label{symdefeq}\overline{\Gor(\D_3)}\cap\left(\Gor(\D_2)\cup \Gor(\D_1)\right)=\emptyset \text { and }
\quad \overline{\Gor(\D_2)}\cap\Gor(\D_1)=\emptyset.
\end{equation}
\item In the dominance partial order $P_3\ge P_2\ge P_1$. We have that no subfamily $\AAA(w)\subset \Gor(\D_i)$ can specialize to an element $\AAA(w_0)$ of $\Gor({D_j})$ having generic Jordan type $P_j$ when $i<j$.
\item Each $\Gor(\D_i)$, $i\in [1,3]$ is an irreducible component of $\Gor(H)$. They have dimensions, respectively $29$, $34$ and $28$ for $\D_1$, $\D_2$, and $\D_3$. 
\end{enumerate}
\end{theorem}
\par
\begin{proof} Note that these three are the only occurring symmetric decompositions of $H$.  Indeed, the three sequences $H(0)$ run through all the height 2 symmetric Gorenstein sequences below $H$. A height two algebra $C(a)$ cannot have a Hilbert function jump up of two from $H_i(C(a))$ to $H_{i+1}(C(a))$ (Equation \eqref{Tcod2eq}). This implies that the first non-zero $H(a)$ for $a>0$, must begin $(0,1,\ldots)$, and it has center $(6-a)/2$; later summands must have smaller center of symmetry. Inspection of symmetric splittings of $H-H(a)$ now show
these are the only decompositions. \vskip 0.2cm\par
We next determine the generic Jordan types for each decomposition.\vskip 0.15cm\par \noindent
i.a.  {\bf Jordan type for a generic $\A\in\D_1(H)$. }\par
First let ${A(0)=\kk[x,y,z]/\Ann F}$, with ${F=X^{[3]}Y^{[3]}+Z^{[3]}}$, $F\in \D_1$. Then a computation by Macaulay2 shows that the Jordan type of $\ell=x+y+z$ is ${P_{A(0),\ell}=(7,5,3,2,1)}$. To show that this is the Jordan type of a general AG algebra $\A$ having symmetric decomposition $\mathcal{D}_1$, we should note that by Lemma~\ref{SLupperbdlem}  ${P_{A,\ell}}$ is bounded above by $H^\vee=(7,5,4,2)$. The only sequences of integers that are strictly greater than $P_{A,\ell}$ and dominated by $H^\vee$ are $(7,5,3,3)$, $(7,5,4,1,1)$, and $H^\vee$ itself. So if we exclude all these cases, we are done. We first exclude $(7,5,3,3)$, arguing in detail. We then will exclude the last two, an easier check.

Let $\A\in \Gor(\D_1)$ and suppose that $\ell$ in $\m_{\A}$  has Jordan type $P_{\A,\ell}=(7,5,3,3)$. Then we can consider four strings corresponding to a pre-Jordan basis:
\begin{align*}
S_1&=\{1,\ell,\ldots,\ell^6\}\\
S_2&=\{b,\ell b,\ldots,\ell^4b\}\\
S_3&=\{c,\ell c,\ell^2c\}\\
S_4&=\{d,\ell d,\ell^2d\}.
\end{align*}
Since $\ell^6$ is non-zero, we know that ${\ell \in \m_{\A}\setminus\m_{\A}^{\,2}}$, i.e.\ $\ell$ has order one. Also $b$ cannot belong to ${\m_{\A}^{\,2}}$, otherwise $\ell^4b$ would be in ${\m_{\A}^{\,6}=\langle \ell^6\rangle}$. Furthermore, by subtracting scalar-multiples of elements in $S_1$ to the elements of $S_2$ if necessary, we may assume that ${b \in \m_{\A}\setminus\m_{\A}^{\,2}}$. Let us now show that both $\ell$ and $b$ represent non-zero, independent classes in 
\[{Q_{\A}(0)_1=\frac{\m_{\A}}{\m_{\A}^{\,2}+(0:\m_{\A}^{\,5})}}.\]
It is straightforward to see that we cannot have ${\ell\in \m_{\A}^{\,2}+(0:\m_{\A}^{\,5})}$, otherwise $\ell^6=0$. On the other hand, if ${b\in \langle \ell \rangle+\m_{\A}^{\,2}+(0:\m_{\A}^{\,5})}$, say ${b=\alpha\ell+b_1+b_2}$, with ${b_1\in \m_{\A}^{\,2}}$ and ${b_2\in (0:\m_{\A}^{\,5})}$, we would have ${\ell^4b=\alpha\ell^5+\ell^4b_1+\ell^4b_2}$, with $\ell^4b_1$ and $\ell^4b_2$ both in ${\m_{\A}^{\,6}=\langle\ell^6\rangle}$, a contradiction. In a similar way, we can show that all the elements in $S_1$ and $S_2$ represent non-zero, independent classes in $Q_{\A}(0)$. Arguing as we did for $b$, we can assume that ${c,d\in\m_{\A}}$. Moreover, since ${\dim_\kk\m_{\A}/\m_{\A}^{\,2}=3}$, we may assume that ${d \in \m_{\A}\setminus\m_{\A}^{\,2}}$ and ${c \in \m_{\A}^{\,2}\setminus\m_{\A}^{\,3}}$ (we cannot have $c$ in $\m_{\A}^{\,3}$, otherwise $\ell^2c$ would be in ${\m_{\A}^{\,5}=\langle \ell^5, \ell^6, \ell^4b\rangle}$). Let us now show that $c$ represents a non-zero class in 
\[{Q_{\A}(0)_2=\frac{\m_{\A}^{\,2}}{\m_{\A}^{\,3}+\m_{\A}^{\,2}\cap(0:\m_{\A}^{\,4})}},\]
not belonging to the span of the classes of $\ell^2$ and $\ell b$. If we had ${c\in \langle \ell^2, \ell b \rangle+\m_{\A}^{\,3}+\m_{\A}^{\,2}\cap(0:\m_{\A}^{\,4})}$, we could write ${c=\alpha\ell^2 +\beta \ell b +c_1+c_2}$, with ${c_1\in \m_{\A}^{\,3}}$ and ${c_2\in \m_{\A}^{\,2}\cap(0:\m_{\A}^{\,4})}$, and we would have ${\ell^2c=\alpha\ell^4+\ell^3b +\ell^2c_1+\ell^2c_2}$, with $\ell^2c_1$ in $\m_{\A}^{\,5}$ and $\ell^2c_2$ in ${\m_{\A}^{\,2}\cap(0:\m_{\A}^{\,4})}$. The symmetric decomposition $\mathcal{D}_2$ tells us that ${\dim_\kk Q_{\A}(0)_4=3=\dim_\kk\m_{\A}^{\,4}-\dim_\kk\m_{\A}^{\,5}}$. Since 
\[{Q_{\A}(0)_4=\frac{\m_{\A}^{\,4}}{\m_{\A}^{\,5}+\m_{\A}^{\,4}\cap(0:\m_{\A}^{\,2})}},\]
this implies that ${\m_{\A}^{\,4}\cap(0:\m_{\A}^{\,2})\subseteq\m_{\A}^{\,5}}$. so both $\ell^2c_1$  and $\ell^2c_2$ lie in ${\m_{\A}^{\,5}=\langle \ell^5, \ell^6, \ell^4b\rangle}$, again a contraction. So far, we have established that all the elements in $S_1$, $S_2$ and $S_3$ represent non-zero, independent classes in $Q_{\A}(0)$. Since ${\dim_\kk Q_{\A}(3)_1=1}$ and 
\[{Q_{\A}(3)_1=\frac{(0:\m_{\A}^{\,3})}{\m_{\A}^{\,2}\cap(0:\m_{\A}^{\,3})+(0:\m_{\A}^{\,2})}},\]
we see that ${(0:\m_{\A}^{\,3})\setminus \m_{\A}^{\,2}}$ is a non-empty subset of ${\m_{\A}\setminus \m_{\A}^{\,2}}$. Since the classes of $\ell$, $b$, and $d$ span ${\m_{\A}/\m_{\A}^{\,2}}$, we can write ${d=d'+\alpha\ell+\beta b+d''}$, with ${d'\in(0:\m_{\A}^{\,3})}$, ${\alpha,\beta\in\kk}$, and ${d''\in\m_{\A}^{\,2}}$. So we get ${\ell^2d=\ell^2d'+\alpha\ell^3+\beta \ell^2b+\ell^2d''}$, where $\ell^2d'$ lies in ${\m_{\A}^{\,6}=\langle \ell^6\rangle}$ and $\ell^2d''$ lies in ${\m_{\A}^{\,4}=\langle \ell^4, \ell^5, \ell^6, \ell^3b, \ell^4b, \ell^2c\rangle}$, a final contradiction. This shows that $(7,5,3,3)$ cannot be the Jordan type of $\ell$.\par Suppose now that $P_{\A,\ell}=(7,5,4,\ldots)$. Then the set of strings would include 
\begin{align*}
S_1&=\{1,\ell,\ldots,\ell^6\}\\
S_2&=\{b,\ell b,\ldots,\ell^4b\}\\
S_3&=\{c,\ell c,\ell^2c,\ell^3c\}.
\end{align*}
As in the previous case, we can see that all the elements in $S_1$ and $S_2$ represent non-zero, independent classes in $Q_\A(0)$. Also as we did before for $d$, since ${\dim_\kk Q_{\A}(3)_1=1}$, we can write ${c=c'+\alpha\ell+\beta b+c''}$, with ${c'\in(0:\m_{\A}^{\,3})}$, ${\alpha,\beta\in\kk}$, and ${c''\in\m_{\A}^{\,2}}$. But this means that ${\ell^3c=\ell^3c'+\alpha\ell^4+\beta \ell^3b+\ell^3c''}$, where ${\ell^3c'=0}$ and $\ell^3c''$ lies in ${\m_{\A}^{\,5}=\langle \ell^5, \ell^6, \ell^4b \rangle}$, a contradiction. This finishes the proof that the general Jordan type of $\A\in \Gor(\D_1)$ is ${(7,5,3,2,1)}$.
\vskip 0.2cm\par\noindent
{\bf i.b. Jordan type for a generic element of $\Gor(\D_2)$.} We
need only establish that the generic Jordan type on $\Gor(\D_2)$ is
$P_\B=(7,5,3,3)$. Since we have an algebra $B$ of that Jordan type we
need only eliminate strong Lefschetz for algebras in $\Gor(\D_2)$.
This can be done in a similar way to the argument for $\Gor(\D_2)$,
by supposing that we have a set of strings form a pre-Jordan basis including
\begin{align*}
S_1&=\{1,\ell,\ldots,\ell^6\}\\
S_2&=\{b,\ell b,\ldots,\ell^4b\}\\
S_3&=\{c,\ell c,\ell^2c,\ell^3c\}.
\end{align*}
and writing ${c=c'+\alpha\ell+\beta b+c''}$, with
${c'\in(0:\m_{\A}^{\,4})}$, ${\alpha,\beta\in\kk}$, and
${c''\in\m_{\A}^{\,2}}$. Then we see that
${\ell^3c'\in(0:\m_{\A})=\m_{\A}^{\,6}}$ and the argument follows.
\vskip 0.2cm\par\noindent
{\bf i.c. Jordan type for a generic element of $\Gor(\D_3)$.} Taking $\ell=x+y+z$ the Jordan type of $(\C,\ell)$ is $P_\C=(7,5,4,2)=H(\C)^\vee$, so $\C$ is SL. By Lemma \ref{SLupperbdlem} this is the maximum possible Jordan type in $\Gor(H)$, so it is the generic Jordan type on $\Gor(\D_3)$.
\vskip 0.2cm\noindent
{\bf Proof of (ii),(iii).} Immediate from the above and Lemmas \ref{dominancelem} and \ref{trianglelem} (see Figure \ref{symdecfig}).\vskip 0.2cm\noindent
{\bf Proof of (iv): Structure and irreducibility of the symmetric decomposition loci.}\par
(iv)(a.) $\D_1$.  By Lemma \ref{ZTlem}, Equation \eqref{dimZteq}, for a fixed choice of $V\subset R_1$ of dimension two, say $V$,the dimension of the family $Z_{H(0)}(V)$, the variety of quotients of the regular local ring ${\mathsf{k}}\{V\}$ having Hilbert function $H(0)=(1,2,3,4,3,2,1)$ is $n-d=16-4=12$; this is a fibre over the projective plane $\mathbb P^2$ parametrizing $V$. Fix $V$ and choose one such algebra $A(0)$ determined by a polynomial $F(0)$ in ${\kk}_{DP}[X,Y]$: now setting $F=F_0+G_3$ with $ G$ generic in $S_3\mod (R\circ F)_3$, where $S={\kk}_{DP}[X,Y,Z]$, we obtain what is termed a relatively compressed $6-3=3$-modification of  $A(0)$
(Definition \ref{rcdef}, Lemma \ref{maxlem}). Here $G_3$ is parametrized by an open dense in an affine space of dimension $\dim S_3-H_3=10-4$; we may add to $G_3$ an arbitrary element  $G_2\in S_2\mod R\circ F$: giving an additional affine space fibre of dimension 2. We so far have determined a fibred space of dimension $14+8=22$. By Lemma \ref{exoticlem} and Equation \eqref{exotic2eq} the exotic summands involving $Z$ give a fibration of fibre dimension $H(0)_3+H(0)_4=4+3=7$ (adding terms in degrees $4,5$ to $F$ - multiplying an old partial by $Z$). Thus $\Gor(\D_1)$ is irreducible of dimension $29$, with the successive fibration structure just given.\par  
(iv)(b.) $\D_2$. Again by Lemma \ref{ZTlem}, the dimension of the family $Z_{H(0)}(V)$ for $H(0)=(1,2,3,3,3,2,1)$ is $n-d=15-3=12$; the choice of $V\in S_1$ gives 2 more, so $14$.  Given $A(0)$ determined by $F$ we form $F+G_4$ where $G_4$ satisfies the very restrictive condition $\dim_{\kk} R_2\circ G_4 \mod \langle X^2,XY,Y^2\rangle = 1$. Here $G_4 $ must  be viewed mod $\langle R_2\circ F\rangle $, a three dimensional subspace $W$ of $\langle X^4,X^3Y,X^2Y^2,XY^3,Y^4\rangle$. We have $R_2\circ G_4\mod \langle X^2XY,Y^2\rangle=\langle L^2\rangle$ or $La_1, a_1\in \langle X,Y\rangle$ and we may write $G_4=\sum_0^3\alpha_iL^{4-i}$ with $\alpha_i\in {\sf k}[X,Y]_i$ we have $\alpha_2=0$ (to assure $\dim_{\sf k}A_2=4$, otherwise there is no restriction on $G_4$. Counting
dimensions we have two for the choice of $\langle L\rangle$, $1+2+4$ for choice of $\alpha_0,\alpha_1,\alpha_3$, so $9$ (this is a not unexpected codimension three out of the possible choices for an element of $R_4/R_2\circ F$).  
For $G_3$  we add $\dim R_3-4=6$ and for $G_2$ add $\dim R_2-4=2$ giving us $14+9+6+2=31$ so far.  The exotic summands by Equation~ \eqref{exotic2eq} add $H(0)_4=3$, giving a total dimension of $34$ for $\Gor(\D_2)$.\par
\vskip 0.2cm\par
(iv)(c.) $\D_3$. First choose $F_6+F_5$: here $F_6$ is the sum $L_1^{[6]}+L_2^{[6]}$ or $L_1^{[5]}L_2$, and
$F_5=L_3^{[5]}+L_4^{[5]}$, or $L_3^{[4]}\cdot L_4$, where each of $L_3,L_4\in S_1$ are linearly independent from $L_1,L_2$. By $\Pgl(2)$ action we may assume that  the linear forms are $X,Y,Z, X+Y+Z$, essentially choosing 4 points in $\mathbb P^2$, so an eight dimensional choice, but we care about the scale, so total of 11 dimensions for $F_6+F_5$. Now we add
elements $G_i \in \langle R_i, i\in [2,4]\rangle$, so $$F=X^6+Y^6+Z^5+(X+Y+Z)^{[5]}+G_4+G_3+G_2.$$  Note that $V=I_2^\ast=\langle X^2,Y^2,Z^2,(X+Y+Z)^{[2]}\rangle^\perp=\langle (xy-xz,xy-yz)\rangle$. We need $I_2^\ast\circ G_4\subset\langle  X^2,Y^2,Z^2,(X+Y+Z)^{[2]}$, which we can test\footnote{This analysis is similar to one for the dimension of the $A$ component in Example \ref{AA*diffex} p. \pageref{I2argument}.} by $I_2^\ast\circ \langle I_2^\ast\circ G_4\rangle=0$, that is $G_4\in W^\perp,  W=V^2=\langle x^2(y-z)^2,xy(y-z)(x-z),y^2(x-z)^2\rangle$. Since $W$ is a vector space of dimension 3, $W^\perp\subset S_4$ has dimension $(15-3)=12$, but we must consider $G_4\mod \langle X^4,Y^4,z\circ (Z^5+(X+Y+Z)^{[5]}\rangle$, a 3-dimensional subspace of $W^\perp$ so we obtain that $G_4$ is parametrized by an affine space bundle $\mathbb A^9$ over the choice of $F_6+F_5$. Now $G_3,G_2$ may be chosen arbitrarily mod powers of $X,Y,Z, (X+Y+Z)$, so from vector spaces of dimension $(r_2-h_2)+(r_3-h_3)=(2+6)=8$.  By Equation \eqref{exotic2eq} there are no exotic terms.  We have $\dim \Gor(\D_3)=(11+9+8)=28$. \par
(iv) Overview: we have $\dim\Gor(\D_2)=34>\dim \Gor(\D_1)=29>\dim\Gor(\D_3)=28$ showing that, just from a dimension argument, the closure of $\Gor(\D_3)$ cannot be all of  $\Gor(\D_2)$ or $\Gor(\D_1)$. The Jordan types for generic elements of the three $\Gor(\D_i)$ are $P_3>P_2>P_1$ by Lemma \ref{dominancelem};  this in contrast with Lemma \ref{trianglelem} on symmetric decompositions $\D_1\ge \D_2\ge \D_3$ and the irreducibility of each $\Gor(\D_i)$ shows that we have three irreducible components of $\Gor(H)$. However, the statement from the semicontinuity of symmetric decompositions  is stronger: for example, no family of AG algebras $\A\in \Gor(\D_2)$ can specialize to an algebra in $\Gor(\D_1)$!\par
This completes the proof of the Theorem.
 \end{proof}\par\noindent
\begin{note}[Discussion of $H(R/I^2)$ after Remark \ref{Isqrem}]\label{discussnote} Here $H(R/I)=(1,3,4,4,3,2,1)$ of length $|H(R/I)|=18$.\par
 $\begin{array}{c||c|c|}
 R/I&H(R/I^2)&H(I/I^2)\\\hline\hline
 {A}&(1,3,6,10,12,12,{\color{red}{9}},9,6,4)&((0,0,2,6,9,10,{\color{red}{8}},9,6,4)\\\hline
  {\mathcal B}&(1,3,6,10,12,12,{\color{red}{10}},8,5,3,2)&(0,0,2,6,9,10,{\color{red}{9}},8,5,3,2)\\\hline
   {\mathcal C}&(1,3,6,10,12,12,{\color{red}{11}},7,5,3,2)&(0,0,2,6,9,10,{\color{red}{10}},7,5,3,2)\\\hline
 \end{array}
 $
 \vskip 0.3cm
  The lengths $|H(I/I^2)|=3\cdot |H(R/I)|=54$ for each of $I=I_A, I=I_B, I=I_C$.  So according to the discussion of Remark \ref{Isqrem} we would have $C$ possibly most general, but $B$, $A$ incomparable.  To check: the Hilbert functions for generic elements $A,B,C$ of the components of $\Gor(H)$.
  \end{note}\vskip 0.2cm
 The next theorem extends the main results of Theorem \ref{2.1thm} (case $k=2$) to arbitrary $H=(1,3,4^k,3,2,1), k\ge 2.$
 
 \begin{theorem}\label{3compthm} Let $k\ge 2$. Consider the Gorenstein sequence $H(k)=(1,3,4^k,3,2,1)$ of socle degree $j=k+4$
 and length $|H(k)|=4k+10$. It has three symmetric decompositions (a),(b),(c):
  \begin{enumerate}[(i.)] 
\item (a.) A generic algebra in $\Gor(\D_1)$, 
$$\D_1=\left(H(0)=(1,2,3,4^{k-1},3,2,1), H(1)=\cdots =H(k)=0, H(k+1)=(0,1,1)\right)$$
will have Jordan type $P_1=(k+5,k+3,k+1,k-1,2)$   Let $F(k)=X^3Y^{k+1}+Z^3$ The algebra $A(k)=R/\Ann F(k)$, 
$\Ann F(k)=(xz,yz,z^3-x^3y^{k+1},x^4,y^{k+2})$ is in $ \Gor(\D_1)$.\par
  (b.) A generic algebra in  $\Gor(\D_2)$
 $$\D_2=\left(H(0)=(1,2,3^{k+1},2,1),H(2)=(0,1^{k+1})\right),$$
will have Jordan type $P_2=(k+5,k+3,k+1.k+1)$. Let $G=X^{k+4}+Y^{k+4}+(X+Y)^{[k+4]}+Z^4$, 
 then $\mathcal B=R/\Ann G$, $ \Ann G=(zx,zy,xy(x-y),z^4-x^{k+4},x^{k+3}+y^{k+3}-3xy^{k+2})$ is in $\Gor(\D_2)$. \par
 (c.) Finally, a generic algebra in $\Gor(\D_3)$,
 $$\D_3=\left(H(0)=(1,2^{k+3},1), H(1)=(0,1,2^{k},1)\right).$$ 
will have strong Lefschetz Jordan type $P_3= (k+5,k+3,k+2,k)$.
Take $\C=R/\Ann W,  W=X^{k+4}+Y^{k+4}+(X+Y)^{k+3}+Z^{k+3}$. $\Ann W=(xz,yz,xy(x-y),z^{k+3}-x^{k+4},x^{k+4}-y^{k+4})$. Then $\C\in \Gor(\D_3)$. Note that, taking $\ell=x+y+z$ the Jordan type of $(\C,\ell)$ is the SL $P_3$.\par
There are no further symmetric decompositions. \par
\item  In the partial order of Equation \eqref{PODeq} we have $\D_1\ge \D_2\ge \D_3$ (see Figure \ref{symdecfig}), and 
\begin{equation}\label{symdefeq}\overline{\Gor(\D_3)}\cap\left(\Gor(\D_2)\cup \Gor(\D_1)\right)=\emptyset \text { and }
\quad \overline{\Gor(\D_2)}\cap\Gor(\D_1)=\emptyset.
\end{equation}
\item In the dominance partial order $P_3\ge P_2\ge P_1$. We have that no subfamily $\AAA(w)\subset \Gor(\D_i)$ can specialize to an element $\AAA(w_0)$ of $\Gor(\D_j)$ having generic Jordan type $P_j$ when $i<j$.
\item Each $\Gor(\D_i), i\in [1,3]$ is an irreducible component of $\Gor(H)$. They have dimensions, respectively, $8k+13,3k+22$ and $8k+12$ for $\D_1,\D_2$ and $D_3$. 
\end{enumerate}
 \end{theorem}
 \begin{proof} Most of the proof is a straightforward extension of that of Theorem \ref{2.1thm}.  
 However, we verify the dimensions.\vskip 0.2cm\par\noindent
 a. {\it We consider the first two decompositions.} Here $\dim \mathcal D$ has contributions from
 \begin{enumerate}[(i)]
 \item $\mathbb P^2$, choice of $\langle x,y\rangle$);
 \item $|Q(0)|-\nu(Q(0))$, dimension of family of quotients of ${\kk}[x,y]$ having Hilbert function $H(Q(0))$ (Lemma \ref{ZTlem}B);
 \item dimension of modification for $Q(1)$, and $8=(10-4)+(6-4)$ for a general element $G_3+G_2$ in degree 3,2, outside of $A^\vee_{2,3}$ for $H=(1,3,4,4,\ldots)$;
 \item exotic summands (Lemma \ref{exoticlem} and Remark \ref{exoticrem}).
 \end{enumerate}
 For $\mathcal D_1$ this is $2+4(k+2)-4+6+2+(4k-1)=8k+13$. Here the exotics occur in degrees
$4$ to $j-1$ and have dimension $\sum_{i=3}^{j-2}H(Q(0)_i=|Q(0)|-9=4k+8-9=4k-1$ by Equation~ \eqref{exotic2eq}.\par
 For $\mathcal D_2$ we have $2+(3k+9)-3+3+8+H(0)_4=3k+22$ (the exotic contribution is $H(0)_4=3$).\vskip 0.2cm\par\noindent
 (b).{\it We now consider $\mathcal D_3$.}\par
 First choose $F_j+F_{j-1}$: here $F_j$ is the sum $L_1^{[j]}+L_2^{[j]}$ or $L_1^{[j-1]}L_2$, and
$F_5=L_3^{[j-1]}+L_4^{[j-1]}$, or $L_3^{[j-2]}\cdot L_4$, where each of $L_3,L_4\in S_1$ are linearly independent from $L_1,L_2$. As in the case $k=2$ by $\Pgl(2)$ action we may assume that  the linear forms are $X,Y,Z, X+Y+Z$, essentially choosing 4 points in $\mathbb P^2$, so an eight dimensional choice, but we care about the scale, so total of 11 dimensions for $F_j+F_{j-1}$. Now we add
elements $G_i \in \langle R_i, i\in [2,j-2]\rangle$, so $F=F_j+F_{j-1}+\sum_{i=5}^{j-2}G_i+(G_3+G_2).$  Here we must choose $G_{j-2},\ldots, G_5$ so they do not add to the Hilbert function: for example, a generic $G_4$ (relatively compressed modification) would add $(0,0,2,0)$ giving $H=(1,3,6,4,\ldots )$. Note that $V=I_2^\ast=\langle X^2,Y^2,Z^2,(X+Y+Z)^{[2]}\rangle^\perp=\langle (xy-xz,xy-yz)\rangle$. We need $I_2^\ast\circ G_4\subset\langle  X^2,Y^2,Z^2,(X+Y+Z)^{[2]}$, which we can test by $I_2^\ast\circ \langle I_2^\ast\circ G_4\rangle=0$, that is $G_4\in W^\perp,  W=V^2=\langle x^2(y-z)^2,xy(y-z)(x-z),y^2(x-z)^2\rangle$. Since $W$ is a vector space of dimension 3, $W^\perp\subset S_4$ has dimension $(15-3)=12$. This corresponds to $W$ defining a scheme of length 12 the order-two vanishing at 4 points of $\mathbb P^2$ $(1,0,0),(0,1,0),(0,0,1),(1,1,1)$, each local ideal being $\m_{p_i}^2$ of colength 3. But we must consider $G_i\mod \langle 
R_{j-i}\circ (X^j,Y^j,((z,xy)\cap R_{j-2})\circ (Z^{j-1}+(X+Y+Z)^{[j-1]}\rangle$. This is a 3-dimensional subspace of $W^\perp$ when $i=j-2=k+2$ but 4-dimensional when $4\le i\le j-3$ so we obtain that $G_4+\cdots +G_{k+2}$ is parametrized by an affine space bundle $\mathbb A^8+\cdots \mathbb A^8+\mathbb A^9$ of total dimension $8(k-2)+9$ over the choice of $F_j+F_{j-1}$.
 And $G_3,G_2$ may be chosen arbitrarily mod powers of $X,Y,Z, (X+Y+Z)$, so from vector spaces of dimension $(r_2-h_2)+(r_3-h_3)=(2+6)=8$.  By Equation \eqref{exotic2eq} there are no exotic terms.  We have $\dim \Gor(\D_3)=(11+8(k-2)+9+8)=8k+12$. 
 \end{proof}\par
 \begin{question}\label{concordques}[Concordance] Is there a codimension three Gorenstein sequence $H$, and two symmetric decomposition strata $\D_1,\D_2$ such that the partial orders on $\Gor(\D_1),\Gor(\D_2)\subset \Gor(H)$ from Jordan type and from semicontinuity for $\D_1,\D_2$ (Lemma \ref{trianglelem}) are concordant? That is, $\D_1<\D_2$ and likewise $P_1\le P_2$? Or, at least consistent - say $P_{\D_1}=P_{\D_2}$ for a generic linear form?  See Question \ref{posetDques}.
\end{question}

\subsection{Infinite families of $\Gor(H)$ having at least two irreducible components.}\label{infsec}
We give an infinite series of examples of codimension three families $\Gor(H)$ with several irreducible components - just below and Theorem \ref{infinitethm}, where $H=(1,3,4^k,3^s,2,1)$ with $k\ge 1,s\ge 2$. Let $R={\kk}[x,y,z]$. First, recall that a Gorenstein algebra $A=R/\Ann F$ of codimension three and socle degree $j$ is a \emph{connected sum} of a socle-degree $j$algebra $B={\kk}[x.,y]/\Ann G, G\in {\kk}[X,Y]$ of codimension 2, and $C={\kk}[z]/\Ann Z^k, k\le j$a of codimension $1$ if we can write $F=G+Z^k$. In that case, given $\ell\in \langle x,y\rangle$ the JT of $A$ with respect to $\ell+z$ is obtained by adjoining the part $(k-1)$ to $P_{B,\ell}$, if $k<j$, and by adjoining $(k-2)$ to $P_{B,\ell}$ if $k=j$.
\begin{example}[Lefschetz property contributing to component structure of $Gor (H)$]\label{2.4ex} We take $H(A)=(1,3,4,4,3,3,3,2,1)$ which permits two symmetric decompositions.
Let $F=X^{[8]}+Y^{[8]}+(Z)^{[7]}+(X+Y+Z)^{[4]}$ (a power sum), and  $\A=R/\Ann F$, of Hilbert function $H(\A)=(1,3,4,4,3,3,3,2,1)$ and symmetric decomposition $\D_1=\D_\A$,
$$\D_1=\left( H_\A(0)=(1,2,2,2,2,2,2,2,1), H_\A(1)=(0,1,1,1,1,1,1), H_\A(3)=(0,0,1,1)\right).$$  Here even $A^\ast$ is SL, as it is readily seen to have Jordan type $P_\ell, \ell=(x+y+z)$ satisfying $P_\ell=H(A)^\vee=(9,7,6,2)$. \par
Now let $G=X^{[8]}+Y^{[8]}+(X+Y)^{[8]}+Z^{[4]}$, and $B=R/\Ann G$. Then $\D_2=\mathcal{D}_B$ satisfies
$$\D_2=\left(H_B(0)=(1,2,3,3,3,3,3,2,1),H_B(4)=(0,1,1,1)\right),$$
of the same Hilbert function $H(B)=H(\A)$. However, since $G$ is a connected sum of components of different degrees, it follows that the Jordan type of $B$ is the concatenation of that for $Q(0)=R/\Ann G_8$, which is $(9,7,5)$ and that for $Q(4)$, which is $3$. So
$P_{B,\ell}=(9,7,5,3)<H_B^\vee$, and $B$ is not SL. \vskip 0.2cm\par\noindent
{\it Specialization: $\Gor(H)$ has two irreducible components}: Since no element of $\Gor(\D_2)$ is strong Lefschetz, no subfamily of $\Gor(\D_2)$ can specialize to a SL element of $\Gor(\D_1)$ (Lemma \ref{dominancelem}).  But no family of algebras of decomposition $\mathcal{D}_1$ can specialize to one of decomposition $\mathcal{D}_2$, since $H_0(B)_2>H_0(A)_2$ (Corollary~\ref{D0cor}).\par
Here the two-component structure does not follow simply from the semicontinuity of $\mathcal{D}$ (see Lemma \ref{trianglelem} and \cite[\S 4]{I2}) but uses the Jordan type in an essential way.
\vskip 0.2cm\par\noindent
{\it Parametrization  of algebras having decomposition $\D_1$.} This should be an irreducible family.
For the associated graded algebra,  we choose two variables to work in (a choice parametrized by $\mathbb P^2$); then for $F_8$ the sum of (eighth powers of) two linear forms, in the two variables, up to constant linear multiple $ (2+1)$, then a general linear form in three variables (with constant coefficient mattering) for $F_7$ and similarly for $F_8$, for a total of  $2+3+3+3=11$ dimensions. The dimension of $Gor(\mathcal{D}_A)$ may be somewhat larger.

\vskip 0.2cm\noindent
{\it Parametrization of algebras having decomposition $\D_2$:} It would appear that the family of AG algebras of decomposition
$\mathcal{D}_B$ should be irreducible, of dimension at least that of the family of stratified $B^\ast$ (stratified by sequences $C(0)\supset C(1)\supset \cdots $ having the same Hilbert functions of quotients $H(a)=H(Q(a))$).  Here
$Q(0)=R/\Ann G_8$, is parametrized by $G_8$ having Hilbert function $H_B(0)$.
For $G_8$ choose a subspace of two variables $\langle x,y\rangle $ from three, then take a sum of three eighth powers of linear forms, with coefficients, mod constant multiple of all: so $G_8$ is parametrized by a bundle of 
dimension $(3(1)+2)$ over $\mathbb P^2$, so has dimension 7. Then one needs to choose a single linear form in three variables for $G_4$, where the constant multiple matters, so $3$ more, so total dimension 10 for the family of associated graded algebras having decomposition $\mathcal{D}_B$.  But the family of AG algebras with decomposition $\mathcal{D}_B$ could be larger, as $G$ may have terms that are hidden. \vskip 0.2cm\noindent
\end{example}
The following result generalizes the Example \ref{2.4ex} and has essentially the same proof.
\begin{theorem}\label{infinitethm} Let  $H=(1,3,4^k,3^s,2,1)$ where $k\ge 1,s\ge 2$, of socle degree $j=k+s+3$ and $F=X^{[j]}+Y^{[j]}+(Z)^{[j-1]}+(X+Y+Z)^{[k+2]}$, determining an algebra $A=R/\Ann F$ of symmetric decomposition 
$$\mathcal D_1=\left(H(0)=(1,2^{j-1},1), H(1)=(0,1^{j-2}), H(j-k-3)=(0, 0,1^{k})\right),$$ of Hilbert function $H$.  The algebra $A$, hence general enough algebras of decomposition $\D_1$ are strong Lefschetz for $\ell=(x+y+z)$ -
they have JT $H^\vee=(j+1,j-1,j-2,k)$. \par
Now let $G=X^{[j]}+Y^{[j]}+(X+Y)^{[j]}+Z^{[k+2]}$, and $B=R/\Ann G$, where $\mathcal{D}_B$ has  symmetric decomposition
$$\D_2=\left(H_B(0)=(1,2,3^{j-3},2,1),H_B(j-k-2)=(0,1^{k+1})\right),$$
of the same Hilbert function $H(B)=H(\A)$. However, since $G$ is a connected sum of components of different degrees, it follows that the Jordan type of $B$ is the concatenation of that for $Q(0)=R/\Ann G_j$, which is $(j+1,j-1,j-3)$ and that for $Q(j-k-2)$, which is $(k+1)$. So
$P_{B,\ell}=(j+1,j-1,j-3,k+1)<H^\vee$. Thus, general elements of $\Gor(\D_2)$ are not weak Lefschetz, but not strong Lefschetz. \par 
A family of elements of  $\Gor(\D_1)$ cannot specialize to any element of  $\Gor(\D_2)$ (Corollary~\ref{D0cor}), and a family of elements of $\Gor(\D_2)$ cannot specialize to an element of $\Gor(\D_1)$ that is strong Lefschetz (Lemma \ref{dominancelem}). These strata comprise two irreducible components of $\Gor(H)$ (for certain $H$ there may be more, see Theorem \ref{2.1thm}).
\end{theorem}
 \begin{remark}  It is easy to set up families of codimension three Gorenstein sequences of higher Sperner number (maximum value of $H$)
 where similar arguments should work to show that $\Gor(H)$ has at least two irreducible components. For example, let $H=(1,3,5^k,4^s,3,2,1)$ with $k\ge 2$ of socle degree $j=k+s+4$ and length $5k+4s+10$, and note that $H^\vee=(k+s+5,k+s+3,k+s+2,k+s,k)$. Let
 $$\D_1=\left(H(0)=(1,2,3^{k+s+1},2,1), H(3)=(0,0,1^{k+s}), H(s+3)=(0,1^{k+1})\right),$$
 and note that $P_{\D_1}=(k+s+5,k+s+3,k+s+1,k+s,k+1)$. Let
 $$\D_2=\left(H(0)=(1,2,3,4^{k+s-1},3,2,1), H(s+3)=(0,1^{k+1},H(k+s)=(0,0,1)\right),$$
 and note that $P_{\D_2}=(k+s+5,k+s+3,k+s+1,k+s-1,k+1,1)$ which is less than $P_{\D_1}$, so again the expected Jordan types under dominance partial order are opposite to the symmetric decomposition order. 
 \end{remark}
\begin{question}\label{unboundedques} Can we get an unbounded number of irreducible components of $\Gor(T)$ as a function of socle degree $j$, in the limit as $j\to \infty$? What is the function $b(j,r)=$ max \# components of 
an AG algebra of socle degree $j$ and embedding dimension $r$? For $r=2, b(j,2)=1$ for all $j$ (since $T$ determines $\D$ and $\Gor(T)$ is irreducible in codimension two). Can we show it goes to $\infty$ for $r=3$ and $j$ large?
\end{question}
\begin{question}\label{fibreques} The map $A\to A^\ast$ for algebras $A\in \Gor(T)$ induces an algebraic map of schemes $\pi:\Gor(T)\to G(T)$ parametrizing graded algebras of Hilbert function $T$. Can we conclude $Gor(\mathcal{D})$ is irreducible if the image $\pi (Gor(\D))\subset G(T)$ is irreducible and there is a sufficiently strong conditions on the fibres?
\end{question}
\subsection{Realizability.}\label{realizablesec} 
We assume in this section that $A$ is a local Artinian Gorenstein algebra of socle degree $j$, so $A=R/I$ where
$R$ is a standard graded regular local ring. Recall that the Hilbert function $H(A)$ for such $A$ satisfies three conditions:\begin{enumerate}[(i).]
\item  $H=\sum _{a=0}^{j-2} H(a)$, where each $H(a)$ is symmetric about $(j-a)/2$.
\item  For each $a, 0\le a\le j-2$ we have that $\sum_{u=0}^a H(u)$ is an O-sequence (possible for a Hilbert function of an Artinian algebra).
\item The symmetric sequence $H(0)$ occurs as the Hilbert function of a graded Gorenstein algebra.
\end{enumerate}
Here the first and third conditions are from Lemma \ref{symmdecomplem} and Equation \eqref{strataeq}; the second condition results from  $H(A^\ast/\mathcal C(a+1))=\sum_{u=0}^a H(u)$.
It is natural to ask if there are further conditions on the symmetric decompositions $\mathcal{D}(A)$ so on the component Hilbert functions $H(A)$ for AG algebras. If it actually occurs as the complete Hilbert function decomposition for an AG algebra, we call $\mathcal{D}(A)$ \emph{realizable}.  S. Masuti and M. Rossi showed that all socle degree four potential decomposition sequences $\mathcal{D}$ satisfying the three conditions are realizable \cite[\S 3.6,3.7]{MR}.\par 
In \cite[p. 6]{IM} we introduced the concept of non-ubiquity: a realizable decomposition with $H(a)\not=0$ is non-ubiquitous if there is some $b<a$ such that
$\mathcal{D}_{\le b}=\left( H(0), H(1),\ldots, H(b)\right)$  is not realizable. We could only show partial non-ubiquity, where we imposed an additional condition on $B$ \cite[\S 8]{IM}. We here exhibit non-realizable sequences  in codimension three, for socle degree $j\ge 8$, satisfying the three conditions. A realizable complete intersection or Gorenstein algebra $A$ that is non-ubiquitous cannot arise from our examples. We were led to our non-realizability results in response to a question asked us  by S. Masuti, and another about the - unknown - Hilbert functions of codimension three complete intersections, by M.~Rossi at a talk in the IIT Bombay Virtual Commutative Algebra seminar of December 29, 2020. Subsequently to our examples, S. Masuti has extended the all-decompositions-realizable result of \cite{MR} to obtain partial results in socle degree $5$ \cite{Mas}. We now give a general construction
of some non-realizable candidate symmetric decompositions.
\begin{theorem}\label{non-realizablethm} Let $R={\mathsf{k}}[x_1,\ldots, x_r], S={\mathsf{k}}_{DP}[X_1,\ldots X_r$, $R'=R[w], S'=S[W]$. Let $B$ be a graded Gorenstein algebra of Hilbert function $H(0)$ of socle degree $j$, and assume that $B=R/I_B$ where $ I_B=\Ann f\cap R$ for an element $f\in S_j$. Assume further that $j\ge 5$ and that $I_B$ is generated in degrees less or equal $j-3$.\footnote{In effect, $j\ge 8$}. Let $H(1)=(0,1,0,\ldots, 0,1_{j-2},0)$, and assume that $H=H(0)+H(1)$ satisfies the Macaulay conditions.  Then $\mathcal{D}=\left(H(0),H(1)\right)$ is not realizable as the (complete) decomposition sequence of an AG algebra quotient $A$ of $R'$ having $Q(0)$ isomorphic to $B$. Nor can it occur as part of a decomposition sequence $\mathcal{D}$ for an AG quotient of $R$ having $Q(0)$ isomorphic to $B$.
\end{theorem}
\begin{proof}
Assume by way of contradiction that $A=\Ann F, F=f+f_{j-1}+\cdots $ has decomposition $\mathcal{D}$. By \cite[Lemma 1.40]{IM}, we have $f_{j-1}$ is linear in $W$, so $f_{j-1}=WG$ where $G\in S_{j-2}$; also $G\not=0$, since $w\circ f_{j-1}=G\mod S'_{\le j-2}$ must not be in $R_1\circ f$, in order to achieve $H(1)_{j-2}=1$. Also, we have $I_u\circ G=0$ for $u\le j-3$, since otherwise in the module $R\circ F$ we would have an extra element $W(h\circ G)$ with $h\circ G\in R_1$, leading to a Hilbert function with $H_2\ge H(0)_2+1$, contradicting the assumption that $\mathcal{D}$ is the complete symmetric decomposition of $H(A)$.\par
Note that $w\circ F=G+w\circ (f_{j-2}+\cdot  +f_2$. We must have $I_{j-2}\circ G\not=0$, in order to obtain $W\in R\circ F$, (the element $W$ must arise from $R/I\circ G$ as $H(1)_1=1$ (the element does not arise from $R'\circ u_u\le u-2$, which would be in a higher $H(a), a>1$).  Since $I_{j-2}\subset (I_{j-3})$ (the ideal) by our assumption on the generators, we have that $I_{j-2} \circ G=0$, a contradiction.  We conclude that $\mathcal{D}$ is not realizable as the decomposition sequence of an AG algebra $A$ having $Q(0)\cong B$.
\end{proof} \par
The following corollary is immediate.
\begin{corollary}\label{nr1cor} Let $H(0)$ be a symmetric Gorenstein sequence of socle degree $j$ such that all graded AG algebras of Hilbert function $H(0)$ are defined by ideals that are generated in degree no more than $j-3$.
Let $H(1)=(0,1,0,\ldots, 0, 1_{j-2},0)$. Then the decomposition $\mathcal{D}=\left(H(0),H(1)\right)$ is not realizable
as the complete symmetric decomposition of an AG algebra, nor can it occur as the first part of a symmetric decomposition $\mathcal{D}$ of some Gorenstein sequence $H$. \end{corollary}
Recall that the compressed Gorenstein Hilbert function of a socle degree $j$ quotient of $R$ has Hilbert function
$H$ satisfying $H_i=\min \{\dim_{\mathsf{k}} R_i, R_{j-1}\}$, for $0\le i\le j$.
\begin{corollary}\label{nr2cor} Let $R={\mathsf{k}}[x,y], R'={\mathsf{k}}[x,y,w]$, and suppose that either\begin{enumerate}[(i).]
\item $H(0)$ is the compressed Hilbert function for $R$ of socle degree at least $8$. OR
\item  $H(0)$ has generator degrees  $(a,b), a+b=j+2, 5\le a\le b$. 
\end{enumerate}
Then the decomposition $\mathcal{D}=\left(H(0), H(1)\right), H(1)=(0,1,0,\ldots, 0, 1_{j-2},0)$ is not realizable, nor can it be the first part of a realizable decomposition.
\end{corollary}
Here (i) is a special case of (ii).\par
\begin{proof} Every graded AG algebra of codimension two is a complete intersection (which is true for all AG algebras of codimension two), so this follows from Corollary \ref{nr1cor}.
\end{proof}
\begin{corollary}\label{nr3cor}
Let $R={\mathsf{k}}[x,y,z], R'={\mathsf{k}}[x,y,z,w]$ and let $B=R/J, J=\Ann f, f\in S_j$.
\begin{enumerate}[(i).] 
\item Let $B=R/J$ be a graded complete intersection for which $J$ has generator degrees
$3\le a\le b\le c$ (then the socle degree $j=(a+b+c)-3$). OR
\item Let $B=R/J$ be any codimension three graded AG algebra of order $\nu(J)\ge 5$.
\end{enumerate} 
Then the decomposition $\mathcal{D}=\left(H(0)=H(B), H(1)\right), H(1)=(0,1,0,\ldots, 0, 1_{j-2},0)$ is not realizable as the complete decomposition of an AG algebra $A$ quotient of $R'$ having $Q(0)$ congruent to $B$. In case (ii), $\mathcal{D}$ is not realizable as the initial part of a symmetric decomposition of any such Gorenstein sequence.
\end{corollary}
\begin{proof} The first statement (i) is immediate from the formula $a+b+c=j+3$ and Corollary~\ref{nr1cor}.
The second (ii) uses a result of A. Conca and G. Valla \cite{CV} that specifies the maximum generator sequence of a graded AG algebra of codimension three of Hilbert function $H(0)$ in terms of the second difference $\Delta^2(H(0))$: in particular they show that the maximum degree of a generator of a graded codimension three $B$ having Hilbert function $H$ is bounded above by $j+2-\nu(H)$.   See also \cite[Theorem 5.25, equation 5.3.3]{IK}; this upper bound could also be derived from S. Diesel's analysis of the poset of generator-relation sequences for AG ideals $J$ determining such $B$ of Hilbert function $H$ \cite{Di}).
\end{proof}\par
We assume our complete intersections $B=R/J$ have embedding dimension that of the ring $R$ we are working in, so  $J$ is generated in degree at least two.
\begin{corollary} Suppose that $B$ is one of the following graded CI's.\begin{enumerate}[(i).]
\item A complete intersection of codimension four, have generator degrees at least $(2,2,3,3)$.
\item A complete intersection of codimension at least five.
\end{enumerate}
Then there is no quotient of $R'=R[w]$ having $Q(0)=B$, and satisfying $H(0)=H(B), H(1)=(0,1,0,\cdots ,0, 1_{j-2},0)$, and any $H(a), a\ge 2$.
\end{corollary}
We cannot find a non-ubiquitous example based on these results. \par
\subsection{Further questions.}\label{probsec}
\begin{question} The Jordan type of a direct sum is certainly the concatenation of the Jordan types of the summands. Do we have inequalities for Jordan types for the modules in an exact sequence?   
Also, given the Jordan type for quotients in a filtration of a graded algebra as $\A^\ast$ what can we say about $P_{\A^\ast,\ell}$?  For example, given each, $ P_{Q(a)}$ for $\ell$ what more can we say abou $P_{\A^\ast}$
(see Lemma \ref{comparelem} for a start).\par
A free extension is a special case of an exact sequence of modules, T. Harima and J.~Watanabe have shown that a free extension $C$ of a strong Lefschetz algebras $A,B$ ($C$ has fibre $B$ over $A$, is itself strong Lefschetz, provided that the Hilbert functions $H(A)$ and $H(B)$ are symmetric and $\cha \F=0$. The \cite[Theorem 2.6]{IMM2} gives a new proof of that result provided $\cha \F=0$ or is greater than the sum $j_A+j_B$ of the socle degrees of $A,B$. The proof there uses Clebsch-Gordan, and assumes both $A,B$ are graded.
\end{question}\noindent

\begin{question} 
Can we use the D. Buchsbaum and D. Eisenbud Pfaffian structure theorem for codimension three Gorenstein algebras to study deformations of AG local algebras? Or to obtain the dimension of families? Or is there a poset of generator orders, as in the graded case \cite{Di},\cite[Theorem 5.25v]{IK}? The Pfaffian structure theorem  has the following consequence, well known for $A $ AG graded of codimension three. Here $R$ is standard-graded.
Recall that the order $\nu(I)$ of an ideal $I$ of $R$ satisfies  $I\subset \m^\nu, I \nsubseteq \m^{\nu+1}$.
\begin{lemma} Let $A=R/I$ be a codimension three local Artinian Gorenstein algebra. Then the number of generators of $I$ is at most $2\nu(I)+1$, where $\nu(I)$ is the order of $I$.
\end{lemma}
\begin{proof} The $v$ generators of $I$ are the Pfaffians - square roots of diagonal minors - of an antisymmetric $v\times v$ matrix $M$ with entries in $R$, where $v=2k+1$ is odd.  We can assume that no entry of $M$ is a unit. Then, every $2k\times 2k$ minor has order at least  $2k$, and each Pfaffian has order at least $k$ so $\nu(I)\ge k$. So $v=2k+1\le 2\nu(I)+1$.
\end{proof}
\end{question}
\begin{question} (Inspired by question of M. Rossi on Dec. 29, 2020 at IIT Bombay seminar talk.)\par
(a) what are the set of Hilbert functions of codimension three complete intersections?\par
(b) can we use Jordan type to help us with this?\par
(c) Do we have families $CI(T)$ (complete intersections of Hilbert function $T$) in codimension three having several irreducible components?\par
(d) Are there codimension three Gorenstein local algebras that cannot be deformed to a complete intersection local algebra? All such Gorenstein algebras can be smoothed, this asks about deformations concentrated at a point.
\end{question}
\begin{question} The concept of non-ubiquity, introduced in \cite{IM}, is defined at the start of Section \ref{realizablesec}. Examples of partial non-ubiquity are in \cite{IM}: there we impose further conditions on the dual generators $F$ for $\mathcal D_{\le k}$.
It remains to determine, in the lowest codimension possible, a non-ubiquitous symmetric decomposition of some Gorenstein sequence $H$.
\end{question}
\begin{question}\label{Jordanordertypeques} We can define Jordan type of an Artinian algebra  $A$ with respect to any linear form $\ell$. We can define a Jordan degree type (specify the initial degrees of each string) for a pair $(A,\ell)$ provided $A$ is graded \cite[\S 2.6]{IMM1}. This could be extended to a local algebra $A$ provided each $\ell$-string has successive terms (beads) whose orders are in adjacent degrees: but this simple behavior does not in general occur (see also Question \ref{Jordanbasisques}). Some of the issues are seen in our Example \ref{2to3ex}. We can of course define JDT for $A^\ast$ and for each $Q(a)$. We discuss these and a sequential Jordan type (let $A=R/I$ consider $R/(I+\m^i)$ for each positive integer $i$) in a work in progress with J.~Steinmeyer \cite{IMS}.
\end{question}

As mentioned, very little is known about the poset $\mathfrak D(H)$ of realizable symmetric decompositions of a Gorenstein sequence $H$ of codimension three or more.  In \cite[\S 5]{I1} an attempt is made to write a complete list for codimension three, for length $|H|\le 21$; and for higher codimension for length $n\le16$.
\begin{question} Can we find an example of $P_A\not=P_{A^\ast}$, using $P_{A^\ast}\le  P_c(H)$ (Definition~\ref{def:relativeLef}i. and Lemma \ref{contiglem})?
\end{question}
 For such an example we would want $H(0)$ to be non-unimodal: this in practice requires codimension of $Q(0)$ to be at least five, as graded Gorenstein sequences of codimension three are unimodal, and the existence of non-unimodal graded Gorenstein sequences is open in codimension four. This may be related to the open question of whether we can find a strong-Lefschetz AG algebra, such that $Q(0)$ is not strong-Lefschetz (Question \ref{Steinques}a,).\par 
\begin{question} Can we understand better the jumping behavior of the order within strings of non-graded AG algebras? Given, say, $\mathcal D(A)$. Can we find an AG algebra for which $P_{A,\ell}>P_c(H)$? This is a weaker ask than determining a strong Lefschetz AG algebra $A$ such that $A^*$ is not SL (Question \ref{Steinques}b) but is also open.
\end{question}
The first part (a) of the next Example illustrates a jumping behavior that is forced by the combination of $\mathcal D(A)$ and the generic Jordan type.  Later examples (b),(c) are related to comparing $P_c(H), P_c(\D)$ (Definition \ref{def:relativeLef}) to generic Jordan types of an AG algebra and its associated graded algebra. \par\vskip 0.2cm\noindent
{\bf Question}. Can we find a non-unimodal $H$ and an algebra $K$ such that $P_{K,\ell}=P_c(H)$ and $Q_K(0)$ is strong Lefschetz?\par
Choosing the highest degree term of the dual generator in ${\sf k}[X,Y]$, as we do in the next Example, guarantees that
$Q_K(0)$ is strong Lefschetz. However we do not resolve the Question.
\begin{example}[Jumping behavior]\label{jumpex}\par
(a). Let  $R={\kk}[x,y,z,w]$, $H_1=(1,4,3,4,2,1)$, 
$F=X^5+Y^5+(X+Y)^5+Z(X^2Y-XY^2)+W^2$.\footnote{This $F$ was introduced in \cite[p. 100]{I2} to show that $H_1=(1,4,3,4,2,1)$ is a Gorenstein sequence; there, it is also shown that $H_1$ has the unique symmetric decomposition $\D_1$.} Then $F\in \Gor(\D_1),$
\[
\D_1=\bigl(H(0)=(1,2,3,3,2,1), H(1)=(0,1,0,1,0), H(3)=(0,1,0)\bigr).
\]
Let $\ell=x+y+z+w$. Then $P_c(H)=(6,4,3,1,1)$ but $A=R/\Ann F$ satisfies
$$P_c(H)>P_{A,\ell}=(6,4,2^2,1) >P_{A^\ast,\ell}=(6,4,2,1^3)=P_c(\D_1).$$
 We 
have $I=( zw, yw, xw, z^2, w^3, x^2z+xyz+y^2z, 2xz+yz-2x^2y-2xy^2+3y^3, 
xz+yz-x^3+y^3, w^2+xy^2z,2xyz+3y^2z+y^4)$, and $I^\ast=(w^2, zw, yw, xw, 
z^2, yz, xz, xy^3-x^2y^2, x^3y-xy^3, x^4-3xy^3+y^4)$.\par
We now show that a balancing jump in the order of beads on one of the strings is forced by the 
$2^2$ in $P_{A,\ell}$. The following is a pre-Jordan basis for 
multiplication by $\ell$, compatible with the Hilbert function (strings 
are arranged in rows; columns correspond to orders, Definition~\ref{orderdef}):
\[
\begin{matrix}
\mathrm{order} &0 &1 &2 &3 &4 &5 \\
\hline
S_1 &1 &\ell &\ell^2 &\ell^3 &\ell^4 &\ell^5\\
S_2 & &x &\ell x &\ell^2x &\ell^3x&\\
S_3 & &z & &\ell z \\
S_4 & & &y^2 &\ell y^2 \\
S_5 & &w
\end{matrix}
\]
This pre-Jordan basis cannot be partitioned into bases for each 
$Q(a)_i$, but the symmetric decomposition helps to explain the jump in 
order that we see in $S_3$. Note that $z$ has order one, but since the 
partial ${z\circ F=X^2Y-XY^2}$ has degree $3$, we have that 
${z\in\m\cap(0:\m^4)}$ and it represents a non-zero class in $Q(1)_1$. 
Naturally, ${\ell z\in\m^2\cap(0:\m^3)}$ (see Remark \ref{multrem}). But 
since $Q(1)_2$ is zero, this forces $\ell z$ to belong to 
${\m^3\cap(0:\m^3)+\m^2\cap(0:\m^2)}$. In fact, ${z\ell=xz+yz=x^3-y^3}$, 
making it an element of ${\m^3\cap(0:\m^3)}$, and we can check that it 
represents a non-zero class in $Q(0)_3$. As we mentioned in Remark 
\ref{multrem}, that $z$ represents a non-zero class in $Q(1)_1$ 
means that $\ell z$ cannot represent a non-zero class in $Q(0)_2$ -- if 
it goes to an earlier $a$ in the $Q(a)$ decomposition, it must go to a 
later order.\par In $S_4$, we may observe, from $y^2$ in $Q(0)_2$ a jump in $\ell y^2$ to a later $Q(1)_3$: note that $y^2$ represents a non-zero class in $Q(0)_2$. Now the class of $\ell y^2$ is also non-zero in $Q(0)_3$, but this space is already spanned by the 
classes of the beads ${\ell^3, \ell^2x, \ell z}$ of previous strings, and if we mod out by these elements, we obtain that $\ell y^2$ induces a non-zero class in $Q(1)_3$. 
\par
(b). Let  $R={\kk}[x,y,z,w]$, $H_2=(1,4,5,4,5,2,1), G=X^3Y^3+X^4Z+Y^4W$,  then $B=R/\Ann B$ satisfies $B\in \Gor(\D_2)$,
$$\D_2=(H(0)=(1,2,3,4, 3,2,1), H(1)=(0,2,0,0,2,0), H(2)=(0,0,2,0).$$
We have $P_c(H)=(7,5,4,4,1,1), P_c(\D_2)=(7,5,3,1^7)$. For $\ell=(x+y+z+w)$ we have
 $$P_c(H)>P_{B,\ell}=(7,5,3,3,3,1)> P_{B^\ast,\ell}=(7,5,3,2^3,1)>P_c(\D_2).$$ 
 Note that adding further degree 4 terms to $G$ does not add to $H(1)$ as we already have filled $\langle X^4,Y^4\rangle =k[X,Y]_4\mod R_2\circ G_6$.\footnote{This is similar to \cite[Example 1.34]{IM} with Hilbert function $H=(1,3,5,4,4,2,1)$, decomposition $$D_2^\prime=\left(H(0)=(1,2,3,4,3,2,1), H(1)=(0,1,0,0,1), H(2)=(0,0,2,0)\right),$$ and dual generator $G^\prime=X^3Y^3+Z(X^4+Y^4)$, $I_{G^\prime}=\Ann G^\prime=(z^2, xyz, x^2z-xy^3, y^2z-x^3y, x^4 -y^4), I_{G^\prime}^\ast = (z^2, xyz, x^2z, y^2z, x^4 -y^4, xy^4, yx^4)$, where $H^\vee=(7,5,4, 3,1)$, 
 $B^\prime={\sf k}\{x,y,z\}/I_{G^\prime}$, and
 for $\ell=x+y+z$  we have
 $H^\vee>P_{B^\prime,\ell}=(7,5,3,3,1,1)>P_{{B^\prime}^\ast,\ell}=(7,5,3,2,1,1,1)> P_c(\D_2^\prime)=(7,5,3,1^5).$ Here in (b), adding a term $Y^4W$ for $G$ in place of $Y^4Z$ of $G^\prime$,  we add $(0,1,0,0,1)$ to $H_B(1)=(0,1,0,0,1)$, making a nonunimodal $H(B^\prime)$. The possible structure of $\D_2$ when there are $Z,W$-linear terms in $G$ is described in \cite[Lemma 3.33]{IM}.}\par
(c) Let  $g=X^3Y^3+X^4Z+Y^4W+WZ^2X$: here  $H=(1,4,7,4,5,2,1), \D_g=\left(H(0)=(1,2,3,4,3,2,1), H(1)=(0,2,0,0,2),H(2)=(0,0,4,0)\right)$ and $E=R/\Ann g$ satisfies $I_E=( w^2 ,yz, xyw ,x^2w, z^3, yw+zw-x^3, xz^2-y^4, x^2z-xy^3)$, $I_E^\ast= (w^2 ,yw+zw ,yz, xzw, x^2w, z^3, xz^2 , x^2z, y^5, xy^4, x^4y, x^5)$. Here $P_c(H)=(7,5,4^2,1^4),\, P_c(\D_g)=(7,5,3,1^9)$ and $$P_c(H)> P_E= (7,5,3^3,1^3)>P_{E^\ast}=(7, 5, 3,2^3,1^3)>P_c(\D_g).$$  This is of interest because of a jump implied by $P_E=(7,5,3^3,\ldots)$: the $3$ strings for $\ell=x+y+z+w$ have orders, $(1,2,3), (1,2,4)$ and $(2,3,4)$. 
\end{example}
\begin{question}\label{posetDques}  Given the (non-symmetric) Gorenstein sequence $H$ \begin{enumerate}[(i).]
\item  What is the cardinality of $\mathfrak D(H)$? 
\item Is $\Gor(\D)$ irreducible, say in codimension three? What is the Zariski-closure $\overline{\Gor(\D)}$ in $\Gor(H)$?
\item  For which $H$ is $\mathfrak D(H)$ simply ordered?
\item There is a stratification on $\Gor(H)$ given by the Jordan types $P_{A,\ell}$ of $A$ with respect to a generic linear form $\ell$.  What can we say about the intersections of the two stratifications? 
\item In particular, if $\Gor(\D)$ is irreducible it has a generic Jordan type. Suppose two symmetric components
$\Gor(\D_1), \Gor(\D_2)$ are irreducible and $\D_1> \D_2$, does the Jordan type of $\D_2$ always dominate that of $\D_1$?
\end{enumerate}
\end{question}
We remind the reader of Question \ref{concordques}, which asks if there are examples of codimension three Hilbert functions having two realizable symmetric decompositions, such that $\Gor(H)$ does \emph{not} have several irreducible components corresponding to them.  And Question \ref{unboundedques}, which asks about estimating the number of irreducible components of $\Gor(H)$ in terms of invariants of $H$.

\begin{ack} We are grateful to Chris McDaniel for his comments, questions, and an examples relevant to issues in defining and analogue of Jordan degree type for non-graded Artin algebras. We thank Steven Kleiman for help with references for the flatness result Lemma \ref{flatlem}. We thank Shreedevi Masuti for comments about the realization problem. We thank Johanna Steinmeyer for her reading an early version, and her very helpful suggestions and comments.\par
 The authors are grateful to the series of annual conferences or workshops on Lefschetz Properties In Algebra and Combinatorics, in particular recent meetings at Levico (2018), CIRM Luminy (2019), and Oberwolfach (2020).

The second author was 
partially supported by CIMA -- Centro de Investiga\c{c}\~{a}o em 
Matem\'{a}tica e Aplica\c{c}\~{o}es, Universidade de \'{E}vora, project 
UIDB/04674/2020 (Funda\c{c}\~{a}o para a Ci\^{e}ncia e Tecnologia).
\end{ack}

\end{document}